\documentclass[12pt,a4paper,reqno]{amsart}

\usepackage{amsmath,amsfonts,amssymb,amsthm,color,a4wide}

\usepackage[T1]{fontenc}
\usepackage{mathrsfs}

\newcommand{\fourIdx}[5]{%
 \setbox1=\hbox{\ensuremath{^{#1}}}%
 \setbox2=\hbox{\ensuremath{_{#2}}}%
 \setbox5=\hbox{\ensuremath{#5}}%
 \hspace{\ifnum\wd1>\wd2\wd1\else\wd2\fi}%
 \ensuremath{\copy5^{\hspace{-\wd1}\hspace{-\wd5}#1\hspace{\wd5}#3}%
 _{\hspace{-\wd2}\hspace{-\wd5}#2\hspace{\wd5}#4}%
 }}

\input colordvi

\newcommand{\oldr}{r}
\newcommand{\newr}{r}
\newcommand{\Leb}{{\rm Leb}\,}

\newcommand\deld[1]{}
\newcommand\delr[1]{#1}

 \newcommand \pb{$\hbox{ }$}
\newcommand\wymM{\mathrm{dim}M}

\theoremstyle{plain}
\newtheorem{theorem}{Theorem}[section]
\theoremstyle{remark}
\newtheorem{remark}[theorem]{Remark}

\theoremstyle{plain}
\newtheorem{corollary}[theorem]{Corollary}
\newtheorem{lemma}[theorem]{Lemma}
\newtheorem{proposition}[theorem]{Proposition}
\newtheorem{definition}[theorem]{Definition}
\newtheorem{assumption}[theorem]{Assumption}

\numberwithin{equation}{section}
\newcommand{\lb}{\langle}
\newcommand{\eps}{\varepsilon}
\newcommand{\rb}{\rangle}

\newcommand{\embed}{\hookrightarrow}

\newcommand{\supp}{\,\textrm{supp}\,}

\newcommand{\toup}{\nearrow}

\begin{document}
\title[Stochastic geometric wave equations]{Stochastic geometric wave equations with values in compact Riemannian homogeneous spaces}
\author{Zdzis\l aw Brze\'zniak, Martin Ondrej\'at}
\address{Martin Ondrej\'at: Institute of Information Theory and Automation of the ASCR}
\email{zb500@york.ac.uk,ondrejat@utia.cas.cz}
\thanks{The research of the second named author was supported by the GA\v CR grant No. 201/07/0237,  the pump priming grant of the University of York and by an EPSRC grant number EP/E01822X/1.}

\date{\today}

\maketitle
\begin{abstract}
Let $M$ be a compact Riemannian homogeneous space (e.g. a Euclidean sphere). We prove existence of a global weak solution of the stochastic wave
equation $\mathbf D_t\partial_tu=\sum_{k=1}^d\mathbf D_{x_k}\partial_{x_k}u+f_u(Du)+g_u(Du)\,\dot W$ in any dimension $d\ge 1$, where $f$ and $g$ are continuous multilinear mappings and $W$ is a spatially
homogeneous Wiener process on $\mathbb R^d$ with finite spectral measure. A nonstandard method of constructing weak solutions of SPDEs, that does not rely on martingale representation theorem, is employed.
\end{abstract}
\maketitle

\section{Introduction}\label{sec:intro}

Wave equations subject to random perturbations and/or forcing  have been a subject of deep  studies in last forty years. One of the reasons  for this is that they find  applications in physics, relativistic quantum mechanics or oceanography, see for instance  Caba\~na \cite{Cabana_1972}, Carmona and Nualart \cite{Carm+Nual_1988,Carm+Nual_1988_b}, Chow \cite{Chow_2002}, Dalang \cite{Dal_1999,Dal+Fr_1998,Dal+Lev_2004}, Marcus and Mizel \cite{Marcus+M_1991}, Maslowski and Seidler \cite{Masl+S_1993}, Millet and Morien \cite{Mill+M_2001}, Ondrej\'at \cite{o4,o1}, Peszat and Zabczyk \cite{Pesz+Zab_1997,Pesz+Zab_2000}, Peszat \cite{Pesz_2002},  Millet and Sanz-Sole \cite{Mill+SS_1999}  and references therein. All these research papers are concerned with  equations whose solutions take values in Euclidean spaces. However  many theories and models in modern physics such as harmonic gauges in general relativity, non-linear $\sigma$-models in particle systems, electro-vacuum Einstein equations or Yang-Mills field theory require  the solutions to take values is  a Riemannian
manifold, see for instance  Ginibre and Velo \cite{Gin+Vel_1982} and Shatah and Struwe  \cite{Shatah+Struwe_1998}. Stochastic wave equations with values in Riemannian manifolds were first studied by the authours of the present paper in \cite{Brz+Ondr_2007}, see also \cite{Brz+Ondr_2009},  where the existence and the uniqueness of global strong solutions was proven for equations defined on the one-dimensional Minkowski space $\mathbb{R}^{1+1}$ and arbitrary target Riemannian manifold. In the present paper, we strive to obtain a global  existence result for equations on general Minkowski space $\mathbb{R}^{1+d}$, $d\in\mathbb{N}$,  however, for the price that the target space is a particular Riemannian manifold - a compact homogeneous space, e.g. a sphere.

 Let us first briefly compare our results (obtained in this paper as well as in the earlier one \cite{Brz+Ondr_2007})   with those  for the  deterministic equations. For more details on the latter   we refer the reader to nice surveys on geometric wave equations by  Shatah and Struwe \cite{Shatah+Struwe_1998} and Tataru \cite{Tataru_2004}. Existence and uniqueness of global solutions is known  for the wave equations for an  arbitrary target manifold provided that the Minkowski space of the equation is either $\mathbb R^{1+1}$ or $\mathbb R^{1+2}$, see  Ladyzhenskaya and  Shubov  \cite{lady+Shubov_1981}, Ginibre and Velo \cite{Gin+Vel_1982}, Gu \cite{Gu_1980}, Shatah \cite{Shatah_1988}, Zhou \cite{Zhou_1999},   Christodoulou and Tahvildar-Zadeh \cite{Chr+T-Z_1993} and M\"uller and Struwe  \cite{Muller+Struwe_1996}. In the former case, depending on the regularity of the initial conditions,  the global solutions are known to exist in the weak \cite{Zhou_1999},  respectively  the strong,  sense \cite{Gin+Vel_1982}, \cite{Gu_1980}, \cite{Shatah_1988}.  In the latter case  the existence of global weak solutions has been established in  \cite{Chr+T-Z_1993} and \cite{Muller+Struwe_1996}. In the more interesting and difficult case of  $\mathbb R^{1+d}$ with  $d\ge 3$,    counterexamples have been  constructed, see for instance \cite{Shatah_1988}, \cite{r24}, \cite{Shatah+Struwe_1998},   showing  that smooth solutions may explode in a finite time and  that weak solutions  can be non-unique. Notwithstanding, existence of global solutions can be proven for particular target manifolds, e.g. for compact Riemannian homogeneous spaces, see Freire \cite{Freire_1996}. The aim of the current paper is to consider a stochastic counterpart of Freire's result. In other words, we will  prove  existence of global solutions  for the wave equation with values  a compact Riemannian homogeneous space even if it is subject  to particular (however quite general) random perturbations.

Towards this end, we assume that $M$ is a \textit{compact Riemannian homogeneous space} (see Sections \ref{sec:notation} and  \ref{sec:manifold} for more details) and we consider the following  initial value problem for the stochastic wave equation
\begin{eqnarray}\label{equa1}
&&\mathbf D_t\partial_tu=\sum_{k=1}^d\mathbf D_{x_k}\partial_{x_k}u+f(u,\partial_t u,\partial_{x_1}u,\dots,\partial_{x_d}u)+g(u,\partial_t u,\partial_{x_1}u,\dots,\partial_{x_d}u)\,\dot{W},\\
&&(u(0),\partial_t u(0))=(u_0,v_0)
\label{equa1_i}
\end{eqnarray}
with a random initial data $(u_0,v_0)\in TM$. Here $\mathbf D$ is the connection on the pull-back bundle $u^{-1}TM$ induced by the Riemannian connection on $M$, see e.g. \cite{Shatah+Struwe_1998} and \cite{Brz+Ondr_2007}. In a simpler way, see \cite{Brz+Ondr_2007},
\begin{equation}\label{eqn-acce}
[\mathbf D_t\partial_t\gamma](t) =\nabla_{\partial_t \gamma(t)}(\partial_t \gamma)(t), t\in I
 \end{equation}
 is   the \emph{acceleration} of the curve $\gamma:I\to M$, $I\subset \mathbb{R}$,  at $t\in I$. Note however that deep understanding  of the  covariant derivative $\mathbf{D}$  is not necessary for reading  this paper.  We will denote by $T^kM$, for $k\in\mathbb{N}$,  the vector bundle over $M$ whose fiber at $p\in M$ is equal to $(T_pM)^k$, the $k$-fold cartesian product of $T_pM$. The nonlinear terms $f$ and $g$ in the equation \eqref{equa1} will be assumed to be of the following forms
  \begin{eqnarray}\label{eqn-f}
&&f: T^{d+1}M \ni (p,v_0,\cdots,v_k)\mapsto f_0(p)v_0+\sum_{k=1}^df_k(p)v_k+f_{d+1}(p)\in TM,
\\
&&g: T^{d+1}M \ni (p,v_0,\cdots,v_k)\mapsto g_0(p)v_0+\sum_{k=1}^dg_k(p)v_k+g_{d+1}(p)\in TM,
\label{eqn-g}
\end{eqnarray}
where $f_{d+1}$ and $g_{d+1}$  are continuous  vector field on $M$, $f_0,g_0:M\to\mathbb{R}$ are continuous function and $f_k,g_k:TM\to TM$ , $k=1,\cdots,d$ are continuous vector bundles homomorphisms, see Definition \ref{cvbh}. Finally, we assume that $W$ is a spatially homogeneous Wiener process, see Section  \ref{sec-main}.

The equation (\ref{equa1}) is written in a formal way but  we showed in \cite{Brz+Ondr_2007} that there are various equivalent rigorous definitions of a solution to (\ref{equa1}). In the present paper  we are going to use the one in which, in view of the Nash isometric embedding theorem \cite{Nash_1956}, $M$ is assumed to be
isometrically embedded  into a certain euclidean space $\mathbb{R}^n$ (and so we can identify $M$ with its image). Hence, $M$ is assumed to be a submanifold in $\mathbb{R}^n$, and in this case, we study, instead of (\ref{equa1})-(\ref{equa1_i}), the following classical second order SPDE

\begin{equation}\label{equa2}
\left\{\begin{array}{l}
\partial_{tt}u=\Delta u+\mathbf S_u(\partial_t u,\partial_t u)-\sum_{k=1}^d\mathbf S_u(u_{x_k},u_{x_k})+f_u(Du)+g_u(Du)\,\dot W\\
(u(0),\partial_t u(0))=(u_0,v_0)
\end{array}
\right.
\end{equation}
 where $\mathbf S$ is the second fundamental form of the submanifold $M\subseteq\mathbb{R}^n$.

Finally, we remark that our proof of the main theorem is based on method {(recently introduced by the authors)} of constructing weak solutions of SPDEs, that does not rely on any kind of martingale representation theorem.

\section{Notation and Conventions}\label{sec:notation}

  We will denote by $B_R(a)$,  for  $a\in \mathbb R^d$ and $R>0$,  the open ball in $\mathbb R^d$ with center at $a$ and we put $B_R=B_R(0)$. Now we will list a notation used throughout the whole paper.

\begin{trivlist}
\item[$\bullet$] $\mathbb{N}=\{0,1,\cdots\}$ denotes  the set of natural numbers, $\mathbb{R}_+=[0,\infty)$,
 $\mathrm{Leb}$  denotes the Lebesgue measure,   $L^p=L^p(\mathbb{R}^d;\mathbb{R}^n)$, $L^p_{\textrm{loc}}=L^p_{\textrm{loc}}(\mathbb{R}^d;\mathbb{R}^n)$, $L^p_{\textrm{loc}}$ is a metrizable topological vector space equipped with a natural countable family of seminorms $(p_j)_{j\in\mathbb{N}}$ defined by
\begin{equation}\label{eqn-L^p-p_j}
p_j(u):=\Vert u\Vert _{L^p(B_j)},\qquad u\in L^p_{\textrm{loc}}, \; j\in\mathbb{N}.
\end{equation}
\item[$\bullet$] $W^{k,p}_{\textrm{loc}}=W^{k,p}_{\textrm{loc}}(\mathbb{R}^d;\mathbb{R}^n)$, for $p\in [1,\infty]$ and $k\in \mathbb{N}$,  is the space of all elements $u\in L^p_{\textrm{loc}}$ whose weak derivatives up to order $k$ belong to $L^p_{\textrm{loc}}$.   $W^{k,p}_{\textrm{loc}}$ is a metrizable topological vector space equipped with a natural countable family of seminorms $(p_j)_{j\in\mathbb{N}}$,
\begin{equation}\label{eqn-W^kp-p_j}
p_j(u):=\Vert u\Vert _{W^{k,p}(B_j)},\qquad u\in W^{k,p}_{\textrm{loc}}, \; j\in\mathbb{N}.
\end{equation}
{The spaces $W^{k,2}$ and $W^{k,2}_{loc}$ are denoted by $H^{k}$ and $H^k_{loc}$ respectively}.
\item[$\bullet$] $\mathscr H_{O}=H^{1}(O;\mathbb{R}^n)\oplus L^2(O;\mathbb{R}^n)$ if $O$ is an open subset of $\mathbb{R}^d$ and, for $R>0$,    $\mathscr H_R=H_{B_R}$.

\item[$\bullet$] $\mathscr H=H^{1}(\mathbb{R}^d;\mathbb{R}^n)\oplus L^2(\mathbb{R}^d;\mathbb{R}^n)$ and
 $\mathscr H_{\textrm{loc}}=H^{1}_{\textrm{loc}}(\mathbb{R}^d)\oplus L^2_{\textrm{loc}}(\mathbb{R}^d;\mathbb{R}^n)$,
\item[$\bullet$] $\mathscr D=\mathscr D(\mathbb{R}^d;\mathbb{R}^n)$ is the class of all compactly supported $C^\infty$-class functions
$\varphi: \mathbb{R}^d \to \mathbb{R}^n$,
\item[$\bullet$] Whenever $\mathcal{Y}$ denotes a certain class of functions defined on $\mathbb{R}^d$, by $\mathcal{Y}_{\textrm{comp}}$ we will denote the space of those elements of $\mathcal{Y}$ whose support is a compact subset of $\mathbb{R}^d$. For instance $L^2_{\textrm{comp}}(\mathbb{R}^d)$ and      $H^{k}_{\textrm{comp}}(\mathbb{R}^d)$.
\item[$\bullet$] If $Z$ is a topological space equipped with a countable system of pseudometrics $(\rho_m)_{m\in\mathbb{N}}$, then, without further reference, we will assume that  the topology of $Z$ is metrized by a metric
\begin{equation}\label{eqn-rho}
\rho(a,b)=\sum_{m=0}^\infty \frac1{2^{m}}\min\,\{1,\rho_m(a,b)\}, \; a,b\in Z.
\end{equation}
\item[$\bullet$] By $\mathscr T_2(X,Y)$ we will denote the class of  Hilbert-Schmidt operators from a separable Hilbert space  $X$ to $Y$. By
$\mathscr L(X,Y)$ we will denote  the space of all linear continuous operators from a topological vector space $X$ to $Y$, see \cite[chapter I]{Rudin_1991_FA}. Both these  spaces will be  equipped with the strong $\sigma$-algebra, i.e. the $\sigma$-algebra generated by the family of maps $ \mathscr T_2(X,Y)\ni B\mapsto Bx\in Y$ or $\mathscr L(X,Y)\ni B\mapsto Bx \in Y$, $x\in X$.
\item[$\bullet$] If $(X,\rho)$ is a metric space then we denote by $C(\mathbb{R}_+;X)$ the space of continuous functions $f:\mathbb{R}_+\to X$.      The space $C(\mathbb{R}_+;X)$  in endowed with the metric $\rho_C$ defined by the following formula
\begin{equation}\label{eqn-rho_C}
\rho_C(f,g)=\sum_{m=1}^\infty 2^{-m}\min\,\{1,\sup_{t\in [0, m]}\rho(f(t),g(t))\},\qquad f,g\in C(\mathbb{R}_+,X).
\end{equation}
\item[$\bullet$] If $X$ is a locally convex space  then  by $C_w(\mathbb{R}_+;X)$ we denote the space of all weakly continuous functions $f:\mathbb{R}_+\to X$  endowed   with the locally convex topology generated by the a family $\|\cdot\|_{m,\varphi}$ of  pseudonorms defined by
\begin{equation}\label{eqn-m+varphi}
\|f\|_{m,\varphi}=\sup_{t\in [0, m]}|\varphi(f(t))|,\qquad m\in\mathbb{N},\quad\varphi\in X^*.
\end{equation}

\item[$\bullet$] By $\pi_R$ we will denote various restriction maps to the ball $B_R$,  for example $\pi_R:L^2_{\textrm{loc}}\ni v\mapsto v|_{B_R}\in L^2(B_R)$ or
$\pi_R: \mathscr H_{\textrm{loc}}\ni z\mapsto z|_{B_R}\in\mathscr H_R$.  In the danger of ambiguity we will make this precise.
\item[$\bullet$] {$\zeta$ is a smooth symmetric density on $\mathbb{R}^d$ supported in the unit ball. If $m\in\mathbb{R}_+$, then we put $\zeta_m(\cdot)=m^d\zeta(m\cdot)$. The sequence $(\zeta_m)_{m=1}^\infty$ is called {\it an approximation of identity}.}
\item[$\bullet$] By $\mathscr S$, see  for instance\cite{Rudin_1991_FA},  we will denote the Schwartz space of $\mathbb{R}$- valued   rapidly decreasing $C^\infty$-class  functions on $\mathbb{R}^d$. By $\mathscr S^\prime$ we will denote the space of tempered distributions on $\mathbb{R}^d$, the dual of the space $\mathscr S$.   The Fourier transform,   in the cases of $\mathscr S$,  $\mathscr S^\prime$ as well as $L^2$,  we will  denote by \;$\widehat{}$\;. For example $\widehat\varphi\in \mathscr S$ will denote the Fourier transform of a function $\varphi\in\mathscr S$.

\item[$\bullet$] Given a positive measure $\mu $ on  $\mathbb{R}^d$, we will denote by    $L^2_{\rm {(s)}}(\mathbb{R}^d,\mu)$ the subspace of
$L^2(\mathbb{R}^d,\mu;\mathbb{C})$ consisting of all $\psi $ such that
$\psi=\psi_{\rm {(s)}}$, where  $\psi _{\rm{(s)}}(\cdot)=\overline{\psi(-\cdot)}$.

\end{trivlist}

As mentioned in the Introduction, throughout the whole paper  we will assume that $M$ is a compact Riemannian homogeneous space. In other words,  $M$ is a compact Riemannian manifold representable  as a quotient space $M=G/H$ of a Lie group $G$ by a subgroup $H$. The Moore-Schlafly Theorem, see \cite{Moore+Schl_1980},  yields the existence of a Lie group homomorphism $\rho:G\to\textrm{SO}(n)$ and a $G$-equivariant isometric embedding $\Phi:M\to\mathbb{R}^n$ satisfying $\Phi\circ g=\rho(g)\circ\Phi$ for any $g\in G$. In particular, $M$ is isometrically embedded in the Euclidean space $\mathbb{R}^n$.  Details will be presented in Section \ref{sec:manifold}. We then put

\begin{trivlist}
\item[$\bullet$] $\mathscr H_{\textrm{loc}}(M)=\{(u,v)\in\mathscr H_{\textrm{loc}}:\;\;v(x) \in T_xM\textrm{ for a.e. } x\in M\}$. The strong, resp. weak,  topologies on $\mathscr H_{\textrm{loc}}(M)$, are by definition the traces of the strong, resp. weak topologies on $\mathscr H_{\textrm{loc}}$. In particular, a function $u:[0,\infty)\to \mathscr H_{\textrm{loc}}(M)$ is weakly continuous, iff $u$ is weakly continuous viewed as a $\mathscr H_{\textrm{loc}}$-valued function.
\end{trivlist}

\section{The Wiener process}\label{sec:Wiener}

Given a stochastic basis $(\Omega,\mathscr{F},\mathbb{F},\mathbb{P})$, where $\mathbb{F}=(\mathscr{F}_t)_{t\geq 0}$ is  a filtration,  an $\mathscr S^\prime$-valued process $W=\big(W_t\big)_{t\geq 0}$ is called a  spatially homogeneous Wiener process with a spectral measure $\mu$ which, throughout the paper  we always assume to be positive, symmetric and to satisfy  $\mu(\mathbb{R}^d)<\infty$, if and only if the following three conditions are satisfied
\begin{itemize}
\item $W\varphi:=\big(W_t\varphi\big)_{t\geq 0}$ is a real $\mathbb{F}$-Wiener process,  for every $\varphi\in\mathscr S$,
\item $W_t(a\varphi+\psi)=aW_t(\varphi)+W_t(\psi)$  almost surely for all $a\in\mathbb{R}$, $t\in \mathbb{R}_+$  and $\varphi,\psi\in\mathscr S$,
\item $\mathbb{E}\,\{W_t\varphi_1 W_t\varphi_2\}=t\langle\widehat\varphi_1,\widehat\varphi_2\rangle_{L^2(\mu)}$  for all $t\ge 0$ and $\varphi_1,\varphi_2\in\mathscr S$, where $L^2(\mu)=L^2(\mathbb{R}^d,\mu;\mathbb{C})$.
\end{itemize}

\begin{remark}
The reader is referred to  the works by Peszat and Zabczyk \cite{Pesz+Zab_1997,Pesz+Zab_2000} and Brze{\'z}niak and Peszat \cite{Brz+Pesz_1999}   for further details on spatially homogeneous Wiener processes.
\end{remark}

Let us denote by $H_\mu\subseteq\mathscr S^\prime$ the reproducing kernel Hilbert space of the $\mathscr S^\prime$-valued random vector $W(1)$, see e.g. \cite{DaP+Z_1992}. Then $W$ is an $H_\mu$-cylindrical Wiener process.   Moreover,  see  \cite{Pesz+Zab_1997} and \cite{Brz+Pesz_1999}, then  the following result identifying the space $H_\mu$ is known.

\begin{proposition}\label{prop_H_mu}
\begin{eqnarray*}
{H}_{\mu} &=&\{\widehat{\psi \mu}:\ \psi \in
L^2_{\rm {(s)}}(\mathbb{R}^d,\mu)\},\\
\langle \widehat{\psi\mu},\widehat{\varphi\mu}\rangle_{{H}_{\mu}} &=&
\int _{\mathbb{R}^d}\psi (x)\overline{\varphi (x)}\,d \mu (x),\quad \psi,
\varphi\in L^2_{\rm {(s)}}(\mathbb{R}^d,\mu).
\end{eqnarray*}
\end{proposition}

See \cite{o1} for a proof of the following lemma that states that under some assumptions,  $H_\mu$ is a function space and that multiplication operators are Hilbert-Schmidt from $H_\mu$ to $L^2$.

\begin{lemma}\label{lem-hsop} Assume that  $\mu(\mathbb{R}^d)<\infty$. Then the reproducing kernel Hilbert space $H_\mu$ is continuously embedded in the space $C_b(\mathbb{R}^d)$ and for any $g\in L^2_{\textrm{loc}}(\mathbb{R}^d,\mathbb{R}^n)$ and a Borel set $D\subseteq\mathbb{R}^d$,  the multiplication operator $m_g=\{H_\mu\ni \xi\mapsto g\cdot\xi \in  L^2(D)\}$ is Hilbert-Schmidt. Moreover,  there exists a universal constant $c_\mu$ such that
\begin{equation}\label{boundedness-of-multiplication}
\|m_g\|_{\mathscr T_2(H_\mu,L^2(D))}\leq c_\mu\|g\|_{L^2(D)}.
\end{equation}
\end{lemma}

\section{The main result}\label{sec-main}

Roughly speaking our main result states that for each reasonable initial data the equation (\ref{equa2}) has a weak solution both in the PDE and in the Stochastic  senses. By a weak solution to equation (\ref{equa2}) in the PDE sense    we mean a process that satisfies a variational form   identity with a certain class of test functions. By a weak solution in the Stochastic Analysis   sense   to equation (\ref{equa2}) we mean a stochastic basis,  a spatially homogeneous Wiener process (defined on that stochastic basis) and a continuous adapted  process $z$ such that (\ref{equa2}) is satisfied, see the formulation of Theorem \ref{thm-main} below. We recall that $\mathbf S$ is the second fundamental tensor/form of the isometric embedding $M\subseteq\mathbb{R}^n$.

\begin{definition}\label{cvbh} A continuous map $\lambda:TM\to TM$ is  a vector bundles homomorphisms iff  for every $p\in M$
the map  $\lambda_p:T_pM\to T_pM$ is linear.
\end{definition}

In our two previous papers \cite{Brz+Ondr_2007,Brz+Ondr_2009} we found two equivalent definitions of a solution to the stochastic geometric wave equation \eqref{equa1}: intrinsic and extrinsic. Contrary to those papers,   in the present article we  only deal with the  extrinsic solutions (as they refer to the ambient space $\mathbb{R}^n$). Hence,   since we do  not introduce (neither use) an  alternative notion of an intrinsic solution, we will not use the adjective ``extrinsic''. We will discuss these issues in a subsequent publication.

\begin{assumption}\label{assum-f+g}
We assume that $f_0$, $g_0$ are continuous functions on $M$, $f_1,\dots,f_d$, $g_1,\dots,g_d$ are continuous vector bundles homomorphisms and $f_{d+1}$, $g_{d+1}$ are continuous  vector fields on $M$. For $b\in\{f,g\}$, we set
\begin{equation}\label{eqn-b}
b(p,\xi_0,\dots,\xi_d)=b_0(p)\xi_0+\sum_{k=1}^db_k(p)\xi_k+b_{d+1}(p),\quad p\in M,\quad \big(\xi_i\big)_{i=0}^d \in [T_pM]^{d+1}.
\end{equation}

\end{assumption}

\begin{definition}\label{def-solution}
A system
\begin{equation}\label{def-sol}
\Big(\Omega,\mathscr{F},\mathbb{F},\mathbb{P},W,z\Big)
\end{equation}
consisting of
\begin{trivlist}
\item[(1)] a stochastic basis $(\Omega,\mathscr{F},\mathbb{F},\mathbb{P})$,  \item[(2)] a spatially homogeneous Wiener process $W$ and \item[(3)] an adapted, weakly-continuous $\mathscr H_{\textrm{loc}}(M)$-valued process $z=(u,v)$
\end{trivlist}
       is called a weak solution to equation  \eqref{equa1} if and only if for all $\varphi\in\mathscr D(\mathbb{R}^d)$, the following equalities holds $\mathbb{P}$-a.s., for all $t\geq 0$
\begin{eqnarray}
\langle u(t),\varphi\rangle&=&\langle u(0),\varphi\rangle+\int_0^t\left\langle v(s),\varphi\right\rangle\,ds,\label{sol_weak_1}
\\
\label{sol_weak_2}
\langle v(t),\varphi\rangle&=& \langle v(0),\varphi\rangle
+\int_0^t\left\langle\mathbf{S}_{u(s)}\left(v(s),v(s)\right),\varphi\right\rangle+\int_0^t\left\langle f(z(s),\nabla u(s)),\varphi\right\rangle\,ds
\\
\nonumber
&+&\int_0^t\left\langle u(s),\Delta\varphi\right\rangle\,ds-\sum_{k=1}^d\int_0^t\left\langle\mathbf{S}_{u(s)}\left(\partial_{x_k}u(s),\partial_{x_k}u(s)\right),\varphi\right\rangle+\int_0^t\left\langle g(z(s),\nabla u(s))\,dW,\varphi\right\rangle,
\end{eqnarray}
where we assume that all integrals above are convergent and we use the notation \eqref{eqn-b}.\\
We will say that the system in \eqref{def-sol} is a weak solution to the initial value problem (\ref{equa1})-(\ref{equa1_i}),  with an initial data being a  Borel probability measure $\Theta$  on $\mathscr H_{\textrm{loc}}(M)$, if and only if it is  a weak solution to equation  \ref{equa1} and

\begin{equation}\label{eqn-ic}
 \mbox{  the law  of } z(0) \mbox{ is equal to } \Theta.
\end{equation}
\end{definition}

\begin{theorem}\label{thm-main} Assume that $\mu$ is a positive, symmetric Borel measure  on $\mathbb{R}^d$ such that $\mu(\mathbb{R}^d)<\infty$. Assume that $M$ is a compact Riemannian homogeneous space.  Assume that $\Theta$ is a Borel probability measure on $\mathscr H_{\textrm{loc}}(M)$ and that the coefficients $f$ and $g$ satisfy  Assumption  \ref{assum-f+g}.
hen there exists {a weak solution to the initial value problem (\ref{equa1})-(\ref{equa1_i}) with the initial data  $\Theta$}.
\end{theorem}

\begin{remark}\label{rem-thm-main-uniqueness} We do not claim  uniqueness of a solution in Theorem \ref{thm-main}, cf. Freire \cite{Freire_1996} where uniqueness of solutions is not known in the deterministic case either.
\end{remark}
\begin{remark}\label{rem-thm-main-nonregularity}
Note that the solution  from Theorem \ref{thm-main} satisfies only $u(t,\omega,\cdot)\in H^1_{loc}(\mathbb{R}^d, \mathbb{R}^n)$, $t\geq 0$, $\omega\in\Omega$. Hence, for $d\geq 2$, the function $u(t,\omega,\cdot)$ need not be continuous in general. Nevertheless,  for almost all $(s,\omega)\in\mathbb{R}_+\times \Omega$, the function $g(z(s),\nabla u(s))$ belongs to $L^2_{\textrm{loc}}$ and hence in view of Lemma \ref{lem-hsop} the It\^o integral in \eqref{sol_weak_2} exists.
\end{remark}

\begin{remark}\label{rem-vbh} In the above Theorem we assume that $f_0$ and $g_0$ are real functions and not general
 vector bundles homomorphisms. We do not know whether our result is true under these more general assumptions.
\end{remark}

Theorem \ref{thm-main} states the mere existence of a solution. The next   result tells us that, among all possible solutions, there certainly exists one that satisfies the ``local energy estimates''.

In order to make this precise we  define the following family of energy functions $\mathbf e_{x,T}(t,\cdot,\cdot)$, where $x\in\mathbb{R}^n$, $T>0$ and $t\in [0,T]$,
\begin{equation}
\label{eqn-e_T}
\mathbf e_{x,T}(t,u,v)=\int_{B(x,T-t)}\left\{\frac12|u(y)|^2+\frac12|\nabla u(y)|^2+\frac 12|v(y)|^2+\mathbf{s}^2\right\}\,dy,\;\quad (u,v)\in\mathscr H_{\textrm{loc}}.
\end{equation}
In the above the constant $\mathbf s^2$ is defined by
\begin{equation}
\label{eqn-s^2}
\mathbf s^2=\max\,\{\|f_{d+1}\|_{L^\infty(M)},\|f_{d+1}\|^2_{L^\infty(M)}+\|g_{d+1}\|^2_{L^\infty(M)}\}.
\end{equation}

\begin{theorem}\label{thm-main2} Let $\mu$ be a positive symmetric Borel measure  on $\mathbb{R}^d$ such that $\mu(\mathbb{R}^d)<\infty$, let $\Theta$ be a Borel probability measure on $\mathscr H_{\textrm{loc}}(M)$ where $M$ is a compact Riemannian homogeneous space and let $f$ and $g$ satisfy Assumption  \ref{assum-f+g}. Then there exists a weak solution $(\Omega,
\mathscr{F},\mathbb{F},\mathbb{P},z,W)$ of (\ref{equa1})) with initial data $\Theta$  such that
\begin{equation}
\label{ineq-energy_01}
\mathbb{E}\,\left\{\mathbf 1_A(z(0))\sup_{s\in [0, t]}L(\mathbf e_{x,T}(s,z(s)))\right\}\leq 4e^{Ct}\mathbb{E}\,\left\{\mathbf 1_A(z(0))L(\mathbf e_{x,T}(0,z(0)))\right\}
\end{equation}
holds for every $T\in \mathbb{R}_+$, $x\in\mathbb{R}^d$, $t\in [0,T]$, $A\in\mathscr B(\mathscr H_{\textrm{loc}})$ and every nonnegative nondecreasing function $L\in C[0,\infty)\cap C^2(0,\infty)$ satisfying (for some $c\in\mathbb{R}_+$)
\begin{equation}\label{eqn-L}
tL^\prime(t)+\max\,\{0,t^2L^{\prime\prime}(t)\}\leq cL(t),\qquad t>0.
\end{equation}
 The constant $C$ in \eqref{ineq-energy_01} depends only on $c$, $c_\mu$ and on the $L^\infty(M)$-norms of  $(f_i,g_i)_{i\in\{0,\cdots,d+1\}}$.
\end{theorem}
\begin{remark}
We owe some explanation about the meaning of the energy inequality \eqref{ineq-energy_01}. First of all please note that for $z=(u,v)\in\mathscr H_{\textrm{loc}}$ we have
\begin{eqnarray}
\nonumber
\mathbf e_{x,T}(0,z)&=&\mathbf e_{x,T}(0,u,v)=\int_{B(x,T)}\left\{\frac12|u(y)|^2+\frac12|\nabla u(y)|^2+\frac 12|v(y)|^2+\mathbf{s}^2\right\}\,dy\\
\nonumber
&=&\frac12\vert u\vert_{W^{1,2}(B(x,T))}^2+\frac12\vert v\vert_{L^2(B(x,T))}^2+\frac{T}2\mathbf{s}^2\\
\label{eqn-e_T_02}
&=&\frac12\vert z\vert_{\mathscr H_{B(x,T)}}^2+\frac{T}2\mathbf{s}^2.
\end{eqnarray}
Similarly, we have
for $z=(u,v)\in\mathscr H_{\textrm{loc}}$
\begin{eqnarray}
\label{eqn-e_T_03}
\mathbf e_{x,T}(s,z)
&=&\frac12\vert z\vert_{\mathscr H_{B(x,T-s)}}^2+\frac{T-s}2\mathbf{s}^2.
\end{eqnarray}
Hence, if a system $(\Omega,\mathscr{F},\mathbb{F},\mathbb{P},z,W)$  is a solution to the initial value problem (\ref{equa1})-(\ref{equa1_i}) and  $A\in\mathscr B(\mathscr H_{\textrm{loc}})$ then the inequality \eqref{ineq-energy_01} becomes
\begin{equation}
\label{ineq-energy_02}
\mathbb{E}\,\left\{\mathbf 1_A(z(0))\sup_{s\in [0, t]}L \big( \frac12\vert z\vert_{\mathscr H_{B(x,T-s)}}^2+\frac{T-s}2\mathbf{s}^2 \big)\right\}
\leq
4e^{Ct}\int_{A}\Big[L\big( \frac12\vert z\vert_{\mathscr H_{B(x,T)}}^2+\frac{T}2\mathbf{s}^2\big)\Big] d\Theta(z).
\end{equation}
In particular, if we take a function $L:\mathbb{R}_+\ni t\mapsto \sqrt{t}\in\mathbb{R}_+$, which satisfies the
\begin{equation}
\label{ineq-energy_04}
\mathbb{E}\,\left\{\mathbf 1_A(z(0))\sup_{s\in [0, t]} \big( \frac12\vert z\vert_{\mathscr H_{B(x,T-s)}}^2+\frac{T-s}2\mathbf{s}^2 \big)^{1/2}\right\}
 \leq  4e^{Ct}\int_{A}\Big[\big( \frac12\vert z\vert_{\mathscr H_{B(x,T)}}^2+\frac{T}2\mathbf{s}^2\big)^{1/2}\Big] d\Theta(z).
\end{equation}
\end{remark}

\section{A new It\^o formula}\label{sec-ito}

In general, neither mild nor weak solutions of SPDEs are semimartingales on their state spaces. Hence, if we need to apply smooth transformations, the It\^o formula cannot be applied directly and certain approximations need to be done to justify the formal Ansatz. The aim of this section is to formulate and prove such an Ansatz which is in fact a special form of an It\^o formula. The regularity assumptions on the processes make this a new and hopefully interesting result. It is certainly crucial for our purposes, see the proof of Theorem \ref{thm-main}. In Section \ref{sec:outline} we will formulate a result, see Proposition
\ref{prop-main}, which can be proved by means of this It\^o Lemma. {This result shows the key idea of the main existence result of  this paper.}

To this end, let us introduce the following trilinear form
\begin{equation}\label{forb}
b(u,v,{\varphi})=\int_{\mathbb{R}^d}\langle u(x),v(x)\rangle_{\mathbb{R}^m}\,\varphi(x)\,dx
\end{equation}
defined for $\varphi$, $u$ and $v$ such that the integral on the RHS of \eqref{forb} converges.

\begin{lemma}\label{lem-ito} Assume that $q\in \{1,2\}$.
Let $u$ be an adapted $H^1_{\textrm{loc}}(\mathbb{R}^d;\mathbb{R}^n)$-valued  weakly continuous  process. Let  $v$, resp. $w$ be progressively measurable $L^2_{\textrm{loc}}(\mathbb{R}^d;\mathbb{R}^n)$-, resp. $L^2_{\textrm{loc}}(\mathbb{R}^d;\mathbb{R}^k)$-valued process. Assume that $U$ is a separable Hilbert space.
Assume that   $h_0$ is a  progressively measurable $L^q_{\textrm{loc}}(\mathbb{R}^d;\mathbb{R}^k)$-valued process, $h_1\dots,h_d$ are progressively measurable $L^2_{\textrm{loc}}(\mathbb{R}^d;\mathbb{R}^k)$-valued processes and
  $g$ is  an $\mathscr L(U,L^2_{\textrm{loc}}(\mathbb{R}^d;\mathbb{R}^k))$-valued process   such that $g\xi$ is $L^2_{\textrm{loc}}(\mathbb{R}^d;\mathbb{R}^k)$-valued progressively measurable  for every $\xi\in U$. Assume that these processes satisfy the following integrability condition: for every $T>0$, $\mathbb{P}$-almost surely,
\begin{eqnarray}
&&\int_0^T\left\{\|v(s)\|^2_{L^2(B_T;\mathbb{R}^n)}+\|w(s)\|^2_{L^2(B_T;\mathbb{R}^k)}+\|g(s)\|^2_{\mathscr T_2(U;L^2(B_T;\mathbb{R}^k))}\right\}\,ds<\infty,\\
&&\int_0^T\Big\{\|h_0(s)\|_{L^q(B_T;\mathbb{R}^k)}+\sum_{k=1}^d\|h_k(s)\|_{L^2(B_T;\mathbb{R}^k)}\Big\}\,ds<\infty.
\end{eqnarray}
Assume finally that  for each  $\varphi\in\mathscr D(\mathbb{R}^d)$ and  for every $t\ge 0$, $\mathbb{P}$-a.s.,
\begin{eqnarray}
\label{eqn_u-phi_1}
\langle u(t),\varphi\rangle&=&\langle u(0),\varphi\rangle+\int_0^t\langle v(s),\varphi\rangle\,ds,
\\
\label{eqn_u-phi_2}
\langle w(t),\varphi\rangle&=&\langle w(0),\varphi\rangle+\int_0^t \Big\{\langle h_0(s),\varphi\rangle+\sum_{k=1}^d\left\langle h_k(s),\partial_{x_k}\varphi\right\rangle\Big\}\,ds+\int_0^t\langle\varphi,g(s)\,dW\rangle.
\end{eqnarray}
 Let $Y:\mathbb{R}^n\to\mathbb{R}^k$ be a $C^\infty$-class  function such that $Y^\prime$ is uniformly bounded if $q=2$, or $Y$ and $Y^\prime$ are uniformly bounded if $q=1$. Then for every $t\ge 0$ and each $\varphi\in\mathscr D(\mathbb{R}^d)$, $\mathbb{P}$-a.s.,
\begin{eqnarray}
\label{eqn_B-phi}
b(w(t),Y(u(t)),{\varphi})&=&b({\varphi},w(0),Y(u(0)))+\int_0^t\, b\left(h_0(s),Y(u(s)),{\varphi}\right)\,ds
\\
\nonumber
&+&\sum_{k=1}^d\int_0^t b\left(h_k(s),Y(u(s)),{\partial_{x_l}\varphi}\right)\,ds
\\
\nonumber
&+& \sum_{k=1}^d\int_0^tb\left(h_k(s),Y^\prime(u(s))\partial_{x_k} u(s),{\varphi}\right)\,ds
\nonumber
\\
&+&\int_0^t \, b\left(w(s),Y^\prime(u(s))v(s),{\varphi} \right)\,ds+\int_0^t\, b\left(g(s)\,dW,Y(u(s)),{\varphi}\right).
\nonumber
\end{eqnarray}
\end{lemma}

\begin{proof}[Proof of Lemma \ref{lem-ito}]
 Let  $(\zeta_m)_{m=1}$ be the approximation of identity introduced in Section \ref{sec:notation}.  Let $K_m$ be the convolution operator with the function $\zeta_m$, i.e.  $$K_mh=h*\zeta_m, \;\; \mbox{ for every locally integrable function } h.$$

 Define, for $m\in\mathbb{N}$, $i=1,\cdots,d$ and $t\in [0,T]$,
 \begin{eqnarray*} u^m(t)&=&K_m(u(t)), \quad  v^m(t)=K_m(v(t)),\\
 w^m(t)&=&K_m(w(t)),\quad h^m_i(t)=K_m(h_i(t)),
 \end{eqnarray*}
and   \begin{eqnarray*}
  g^m(t)\xi&=&K_m(g(t)\xi),  \mbox{ for } \xi\in U.
   \end{eqnarray*}
      For any test function $\psi \in\mathscr D(\mathbb{R}^d)$   we put  $\varphi^m=K_m\psi$ and  use  it as a test function in the equalities (\ref{eqn_u-phi_1}-\ref{eqn_u-phi_2}). Since $\langle u,K_m\psi\rangle=\langle K_m u,\psi\rangle$ for all ``good'' functions $u$ and $\psi$, we infer that
  for every $t\ge 0$, the following equalities hold a.s.
\begin{eqnarray}
\label{eqn_u^m-phi_1}
\langle u^m(t),\psi\rangle&=&\langle u^m(0),\psi\rangle+\int_0^t\langle v^m(s),\psi\rangle\,ds,
\\
\label{eqn_u^m-phi_2}
\langle w^m(t),\psi\rangle&=&\langle w^m(0),\psi\rangle+\int_0^t\left\langle h^m_0(s)-\sum_{k=1}^d\partial_{x_k}h^m_l(s),\psi\right\rangle\,ds+\int_0^t\langle\psi,g_m(s)\,dW\rangle.
\end{eqnarray}
Let us now fix a  natural number $j$ such that $j>d/2$ and define  adapted  $H^j_{loc}(\mathbb{R}^d;\mathbb{R}^n)$-, resp. $H^j_{loc}(\mathbb{R}^d;\mathbb{R}^k)$-valued processes $a^m$ and $b^m$  by the following formulae
\begin{eqnarray}\label{eqn_a^m_1}
a^m(t)&=&u^m(0)+\int_0^t v^m(s)\,ds,
\\
\label{eqn_n^m_1}
b^m(t)&=&w^m(0)+\int_0^t\left\{h^m_0(s)-\sum_{l=1}^d\partial_{x_l}h^m_l(s)\right\}\,ds+\int_0^tg_m(s)\,dW.
\end{eqnarray}
Then, from the last four  equalities we infer that $a^m(t)=u^m(t)$, $b^m(t)=w^m(t)$ almost surely for every $t\ge 0$. Indeed, for each $t$ and $\psi$,  $\langle a^(t),\psi\rangle=\langle u^m(t),\psi\rangle$ on a set $\Omega(t,\psi)$ of full measure which however may depend on $\psi$. Taking a countable dense sequence $(\psi_l)$, we infer that $a^m(t)=u^m(t)$ on a set $\bigcap_{l}\Omega(t,\psi_l)$. Note that this set is  of full measure. Let us also observe that both processes $a$ and $b$ are semi-martingales.
Let us now choose $R>0$ such that $\supp \psi \subset B_R$.  Since $j>d/2$, in view of the Sobolev embedding Theorem,   we infer that  a map
\begin{equation}\label{def-Q}
Q:H^{j}(B_R;\mathbb{R}^n)\oplus H^j(B_R;\mathbb{R}^k)\ni (u,v)\mapsto \int_{B_R}\varphi(x)\langle v(x),Y(u(x))\rangle_{\mathbb{R}^k}\,dx \in \mathbb{R}
\end{equation}
is of $C^\infty$-class  and for all $z=(u,v),w=(w_1,w_2),\tilde w=(\tilde  w_1,\tilde  w_2)\in H^{j}(B_R;\mathbb{R}^n)\oplus H^j(B_R;\mathbb{R}^k)$ we have
\begin{eqnarray*}
Q^\prime(z)w&=&\int_{B_R}\langle w_2(x),Y(u(x))\rangle_{\mathbb{R}^k}\,\varphi(x)\,dx+\int_{B_R}\langle v(x),Y^\prime(u(x))w_1(x)\rangle_{\mathbb{R}^k}\,\varphi(x)\,dx,
\\
Q^{\prime\prime}(z)(w,\tilde w)&=&\int_{B_R}\langle w_2(x),Y^\prime(u(x))\tilde  w_1(x)\rangle_{\mathbb{R}^k}\,\varphi(x)\,dx
+\int_{B_R}\langle\tilde  w_2(x),Y^\prime(u(x))w_1(x)\rangle_{\mathbb{R}^k}\,\varphi(x)\,dx
\\
&+&\int_{B_R}\langle v(x),Y^{\prime\prime}(u(x))(w_1(x),\tilde  w_1(x))\rangle_{\mathbb{R}^k}\,\varphi(x)\,dx.
\end{eqnarray*}
Since $Q^{\prime\prime}(z)(w,w)=0$ if $w_1=0$ and since
$$b(w,Y(u),\phi)=Q(w,u),   \mbox{ if }  (w,u)\in H^j(B_R;\mathbb{R}^n)\oplus H^j(B_R;\mathbb{R}^k),$$
by the
 It\^o formula applied to  the $H^{j}(B_R;\mathbb{R}^n)\oplus H^j(B_R;\mathbb{R}^k)$-valued semi-martingale $(a^m,b^m)$ and the function $Q$, we obtain,  for every $t\ge 0$, almost surely,
\begin{eqnarray*}
b(w^m(t),Y(u^m(t)),{\varphi})&=&b(w^m(0),Y(u^m(0)),{\varphi})+\int_0^t\, b\left(h^m_0(s),Y(u^m(s)),{\varphi}\right)\,ds
\\
&+&\sum_{k=1}^d\int_0^t b\left(h^m_k(s),Y(u^m(s)), {\partial_{x_k}\varphi}\right)\,ds
\\
&+&\sum_{k=1}^d\int_0^t\, b\left(h^m_k(s),Y^\prime(u^m(s))\partial_{x_k}u^m(s),{\varphi}\right)\,ds
\\
&+&\int_0^t\, b\left(w^m(s),Y^\prime(u^m(s))v^m(s), {\varphi}\right)\,ds
\\
&+&\int_0^t\, b\left(g^m(s)\,dW,Y(u^m(s)),{\varphi}\right).
\end{eqnarray*}
 Now it suffices to pass with $m$ to infinity.  For this we just need to  consider  two cases.
 \begin{trivlist}
 \item[\textit{Case 1}] $q=2$. Now the boundedness of $\nabla Y$ implies the convergence
 $$Y(u^m(s))\to Y(u(s)) \mbox{ in  } L^2_{\textrm{loc}}(\mathbb{R}^d,\mathbb{R}^k),$$
 for every $s\in [0,T]$.
  \item[\textit{Case 2}] $q=1$. Now the boundedness of $Y$ implies the convergence  $$b(h^m_0(s),Y(u^m(s)),\varphi)\to b(h_0(s),Y(u(s)),\varphi),$$ for every $s\in [0,T]$.
 \end{trivlist}
Simple and  standard but tedious details are omitted. The proof is now complete.

\end{proof}

\section{The target manifold $M$}\label{sec:manifold}

Let $M$ be a compact Riemannian manifold and let the following hypotheses be satisfied:

\begin{itemize}
\item[\textbf{M1}] There exists a metric-preserving diffeomorphism of $M$ to a submanifold in $\mathbb{R}^n$ for some $n\in\mathbb{N}$ (and from now on we will  identify $M$ with its image).
\item[\textbf{M2}] There exists a $C^\infty$-class  function $F:\mathbb{R}^n\to[0,\infty)$ such that $M=\{x:F(x)=0\}$ and $F$ is constant outside some large ball in $\mathbb{R}^n$.
\item[\textbf{M3}] There exist a finite sequence $(A^i)_{i=1}^N$   of skew symmetric linear operators on  $\mathbb{R}^n$    such that for each $i\in\{1,\cdots,N\}$,
\begin{eqnarray}\label{ass-M3_1}
&&\langle\nabla F(x),A^ix\rangle=0, \;\mbox{ for every } x\in\mathbb{R}^n,\\
\label{ass-M3_2}
&&A^ip\in T_pM,\; \mbox{ for every } p\in M.
\end{eqnarray}
\item[\textbf{M4}] There exist a family $\big(h_{ij}\big)_{1\leq i,j\leq N}$ of $C^\infty$-class  $\mathbb{R}$-valued functions on $M$
    such that
\begin{equation}
\label{ass-M4_1}
\xi=\sum_{i=1}^N\sum_{j=1}^Nh_{ij}(p)\langle\xi,A^ip\rangle_{\mathbb{R}^n}A^jp,\;  p\in M,\; \xi\in T_pM.
\end{equation}
\end{itemize}

\begin{remark}\label{rem-hom-space} Let us show that compact Riemannian homogeneous spaces satisfy \textbf{M1} - \textbf {M4}. For this assume that $M$ is a compact Riemannian manifolfd with a compact Lie group $G$ acting  transitively by isometries on $M$, i.e.  there exists a smooth map
\begin{equation}
\label{eqn-pi}
 \pi: G\times M \ni (g,p)\mapsto gp \in  M
\end{equation}
such that, with $e$ being the unit element in $G$,
\begin{trivlist}
\item[\pb] for all  $g_0,g_1\in G$ and $p\in M$,
 $ep=p$ and $(g_0g_1) p=g_0(g_1p)$,
\item[\pb]   there exists (equivalently, for all) $p_0\in M$ such that $\{gp_0:g\in G\}=M$,
\item[\pb]  for every $g\in G$, the map $\pi_g:M\ni p\mapsto \pi(g,p)\in M$ is an isometry.
\end{trivlist}

By   \cite[Theorem 2.20 and Corollary 2.23]{kir}, see also   \cite[Theorem 3.1]{onvi}, for every $p\in M$ the set $G_p=\{g\in G:gp=p\}$ is a closed Lie subgroup of  $G$ and  the mapping
\begin{equation}
\label{eqn-pi^p}
\pi^p:
G\ni g\mapsto gp\in M
 \end{equation}
 is a locally trivial fibre bundle over $M$ with fiber $G_p$. In particular, for every $p\in M$, the map $\pi^p$ is a submersion. Moreover, by the Moore-Schlafly Theorem \cite{Moore+Schl_1980}, there exists an isometric embedding
 \begin{equation}
\label{eqn-Phi}
 \Phi:M\to\mathbb{R}^n,
 \end{equation}
  an orthogonal representation, i.e. a smooth Lie group homomorphism,
  \begin{equation}
\label{eqn-rhoG}
\rho: G \to SO(n),
 \end{equation}
   where $SO(n)$ is the   orthogonal group\footnote{Let us recall that the Lie group $SO(n)$ consists of all $n\times n$ real matrices $A$ such $ A^t A = I$ and  $\det A = 1$ and that  the associated Lie algebra $so(n)$ consists of all $n\times n$ real matrices $A$ such that $A^t = - A$.}, such that  \begin{equation}
\label{eqn-Phi2}
  \Phi(gp)=\rho(g)\Phi(p)   \mbox{ for all } p\in M \mbox{ and }g\in G.
 \end{equation}
We identify each matrix $A\in SO(n)$ with a linear transformation of $\mathbb{R}^n$ (with respect to the canonical ONB of $\mathbb{R}^n$).

Let $\{X_i: i\in I\}$ be a basis in $T_eG$ and let us denote,
  \begin{equation}
\label{eqn-A_i}
A_i=d_e\rho(X_i)\in so(n), \mbox{ for each } i \in I.
\end{equation}

Then for each $i\in I$, $A_i$ is identified with a   skew-self-adjoint linear map in $\mathbb{R}^n$.  Let us also put  $$N=\Phi(M).$$
Then, since as observed earlier  $\pi^p$ is a submersion for each $p\in M$, we infer that
  \begin{equation}
\label{eqn-A_ix}
\mbox{ linspan} \{ A_ix: i\}  
=T_xN,   \mbox{ for every } x\in N.
\end{equation}
Let us choose a smooth function $h:\mathbb{R}^n\to\mathbb{R}_+$ such that $N=h^{-1}(\{0\})$ and $h-1$ has compact support. Let us denote by  $\nu_G$ a probability measure on $G$ that is invariant with respect to the right multiplication.    Then a function $F$ defined by
  \begin{equation}
  \label{eqn-F}
F:\mathbb{R}^n \ni x\mapsto \int_Gh(\rho(g)x)\,\nu_G(dg),\qquad x\in\mathbb{R}^n
\end{equation}
has  the following properties.
\begin{trivlist}
\item[(i)]
The function $F$ is  of $C^\infty$-class,
\item[(ii)] the function  $F-1$ has compact support,
\item[(iii)]   $N=F^{-1}(\{0\})$,
\item[(iv)] the function  $F$ is $\rho$-invariant, i.e. $F(\rho(g)x)=F(x)$ for all $g\in G$ and $x\in\mathbb{R}^n$.
\end{trivlist}
 Hence, for every $i\in I$,  if $x\in\mathbb{R}^n$ and $\gamma:[0,1]\to G$ is a smooth curve  such that $\gamma(0)=e$ and $\dot\gamma(0)=X_i$,  then by the property (iv) above and the chain rule, we have
 \begin{equation*}
0=\frac{d}{dt} F(\rho(\gamma(t))x)\vert_{t=0}= d_xF(A_ix).
 \end{equation*}
 This proves the first one of the two additional properties of the function $F$,
\begin{trivlist}
\item[(v)]  for every $i\in I$ and  $x\in\mathbb{R}^n$, $\langle\nabla F(x),A_ix\rangle=0$,
\item[(vi)] for every $i\in I$ and each $x\in N$, $A_ix\in T_xN$ and   for each $x\in N$,  the set $\{A_ix: i\in I\}$ spans the tangent space $T_xN$.
\end{trivlist}
To prove the first part of (vi) it is sufficient to observe that   if $x\in N$, $i$ is fixed and $\gamma$ is as earlier, then $\rho(\gamma(t))x\in N$ for every $t\in [0,1]$ and so  $A_ix=\frac{d}{dt}\rho(\gamma(t))x\vert_{t=0} \in T_xN$. The second part of (vi) is simply \eqref{eqn-A_ix}.

  In view of \eqref{eqn-A_ix} for each  $x\in N$ we can find $i_1,\cdots,i_{\wymM}$ and a neighbourhood $U_x$ of  $x$ in  $N$ such that ,
  \begin{equation}
  \label{eqn-basis}
A_{i_1}y,\dots, A_{i_\wymM}y \mbox{ is  a basis of } T_yN \mbox{ for each } y \in U_x.
\end{equation}
By the Gram-Schmidt orthogonalization procedure we can find $C^\infty$-class functions $\alpha_{jk}:U_x\to \mathbb{R}$, $j,k=1,\cdots, \wymM$ such that for each $y\in U_x$, the vectors
  \begin{equation}
  \label{eqn-basis2}
Z_j(y)=\sum_{k=1}^{\wymM}\alpha_{jk}(y)A_{i_k}y, \; j=1,\cdots, \wymM
\end{equation}
  form an ONB of  $T_yN$. Hence, if $h_{i_ki_l}:=\sum_{j=1}^{\wymM}\alpha_{jk}\alpha_{jl}$ and $h_{kl}:=0$ for all other indices,  then
\begin{equation}\label{localdevelopment}
\xi=\sum_{j=1}^{\wymM}\langle\xi,Z_j(y)\rangle Z_j(y)=\sum_\alpha\sum_\beta h_{\alpha\beta}(y)\langle\xi,A_\alpha y\rangle A_\beta y,\qquad\xi\in T_yN,\quad y\in U_x.
\end{equation}
 This local  equality  can be extended to a global one  by employing the partition of unity argument.
 Hence we have shown that $M$ satisfies all four  assumptions \textbf{M1}-\textbf{M4}.
\end{remark}

\begin{remark}\label{cor-M3_1+2} Let us note here that the condition \eqref{ass-M3_2} is a consequence of the condition \eqref{ass-M3_1} if the normal space $(T_pM)^\perp$ is one-dimensional (which is not assumed here), e.g. if $M=\mathbb{S}^{n-1}\subset\mathbb{R}^{n}$.
\end{remark}
\begin{remark}\label{rem-cond-M}
 As we have explained in Remark \ref{rem-hom-space},  assumptions \textbf{M1}-\textbf{M4} are satisfied whenever $M$ is a compact Riemannian
homogeneous space, see the original papers   \cite{Shatah+Struwe_1998}, \cite{Helein_1991}, \cite{Freire_1996} or \cite{Moore+Schl_1980}. In particular they are satisfied when $M$ is
a sphere $\mathbb{S}^{n-1}\subset\mathbb{R}^{n}$, see  for instance \cite{Shatah_1988}.
 This can be seen as follows.   For $i,j\in\{1,\cdots,n\}$ such that $i<j$ let $A^{ij}$ be a skew-symmetric linear operator in $\mathbb{R}^n$ whose matrix in the standard basis $\{e_1,\cdots,e_n\}$ of $\mathbb{R}^n$
is  equal to $[a^{ij}_{kl}]_{k,l=1}^n$, where $  a^{ij}_{kl}=0$ unless
   \begin{equation}
   a^{ij}_{kl}=
  \begin{cases}
   1,& \mbox{ if } (k,l)=(i,j), \cr
   -1, & \mbox{ if } (k,l)=(j,i).
   \end{cases}
    \end{equation}
Let a function  $\varphi: \mathbb{R}_+\to \mathbb{R}_+$ be such that $\varphi(x)=0$ iff $x=1$ and $\varphi(x)=1$ iff $x\in [0,\frac12]\cup [2,\infty)$.
    Define then a function $F:\mathbb{R}^n\to\mathbb{R}_+$ by formula
    $$F(x)=\varphi(|x|^2), \; x\in\mathbb{R}^n. $$ Then it is easy to verify that

    \begin{eqnarray}\label{eqn-M3_1-S^n}&&\left\langle\nabla F(x),A^{ij}x\right\rangle=0, \; \mbox{ for every }x \in\mathbb{R}^n,\\
    \label{eqn--M3_2-S^n}
    &&A^{ij}p\in T_p\mathbb{S}^{n-1} \; \mbox{ if } p\in\mathbb{S}^{n-1}\\
    \nonumber\mbox{and} &&\\
  \label{eqn--M4_1-S^n}
&& \xi=\sum_{1\leq i<j\leq n}\left\langle\xi,A^{ij}p\right\rangle A^{ij}p \; \mbox{ if } p\in\mathbb{S}^{n-1}, \; \xi\in T_p\mathbb{S}^{n-1}.
    \end{eqnarray}
    In the case $n=3$ the three matrices $A^{ij}$ can be relabeled as $A_i$, $i=1,2,3$, and, with $e_i$, $i=1,2,3$ being the canonical ONB of $\mathbb{R}^3$, we have   $$A_ip=p\times e_i, \;\; p\in \mathbb{R}^3, \;i=1,2,3.$$
    Let us note that in the very special case the formula \eqref{eqn--M4_1-S^n} takes the following particularly nice form
        \begin{equation}  \label{eqn--M4_1-S^2}
    \xi=\sum_{i=1}^3 \,\left\lb \xi , p\times e_i\right\rb \, (p\times e_i), \;\; \mbox{ if } \vert p \vert=1, \, \lb \xi,p\rb=0.
    \end{equation}
\end{remark}
\begin{remark}\label{rem-cond-M4} Let us  denote by $\tilde{h}_{ij}$ a smooth compactly supported extension of the function $h_{ij}$ to the whole $\mathbb{R}^n$.  For $k\in\{1,\cdots,N\}$ let us define a   map $Y^k:\mathbb{R}^n\to \mathbb{R}^n$
\begin{equation}\label{eqn-tilde-Y^i}
\tilde Y^k(x)=\sum_{j=1}^N\tilde{h}_{kj}(x)A^jx, \; x\in \mathbb{R}^n,\; k\in\{1,\cdots,N\}.
\end{equation}
For each $k\in\{1,\cdots,N\}$ and for all $x\in \mathbb{R}^n$,  $\tilde Y^k(x) $  is a skew symmetric linear operator in  $\mathbb{R}^n$.
Let us also denote by $Y^k$ the restriction of  $\tilde Y^k$, i.e.
\begin{equation}\label{eqn-Y^i}
Y^k(p)=\sum_{j=1}^Nh_{kj}(p)A^jp, \; p\in M.
\end{equation}
In view of assumption \textbf{(M3)}, for each $p\in M$, $Y^k(p)\in T_pM$ and hence $Y^k$ can be viewed as a vector field on $M$. Moreover, the identity   \eqref{ass-M4_1} from Assumption \textbf{M4} can be equivalently expressed in terms of the  vector fields $ Y^k $ as follows
\begin{equation}
\label{ass-M4_2}
\xi=\sum_{k=1}^N \left\langle\xi,A^kp\right\rangle Y^kp,\;  p\in M,\; \xi\in T_pM,
\end{equation}
where\footnote{Note that since the embedding $i: M\embed \mathbb{R}^n$ is isometric, $\langle \cdot,\cdot \rangle$ could also be understood as the scalar product in $T_{p}M$ generated by the Riemannian metric on $M$.} $\langle \cdot,\cdot \rangle=\lb \cdot,\cdot\rb_{\mathbb{R}^n}$.
Identity \eqref{ass-M4_2} is a close reminiscence of formula (7) in \cite[Lem 2]{Helein_1991}.
\end{remark}

The following Lemma will prove most useful in the proof of the existence of a solution. To be precise in the \textit{identification} part of the proof. Let us recall that by $\mathbf{S}$  we denote the second fundamental form of the submanifold $M \subset \mathbb{R}^n$.

\begin{lemma}\label{lem-2ff} For  every $(p,\xi)\in TM$
 we have
\begin{equation}\label{eqn-2ff}
\bold S_p(\xi,\xi)=\sum_{k=1}^N\left\langle\xi,A^kp\right\rangle d_pY^k(\xi),
\end{equation}
where $d_pY^k(\xi):=d_p{\tilde Y^k}(\xi)$ and $d_p{\tilde Y^k}$ is the Fr{\`e}chet derivative of the map $\tilde Y^k$ at $p$.
\end{lemma}

\begin{proof}[Proof of Lemma \ref{lem-2ff}]  Let us denote in this proof by $i: M\embed \mathbb{R}^n$ the natural embedding of $M$ into $\mathbb{R}^n$.
Let us take $p\in M$ and $\xi \in T_pM$. Let $I\subset \mathbb{R}$ be an open interval such that $0\in I$.
Let $\gamma:I\to M$ be a curve  such that $\gamma(0)=p$ and $\dot\gamma(0)={\big( i\circ \gamma\big)}^{\cdot}(0)=\xi$. Then by
identity \eqref{ass-M4_2} we get
\begin{equation}\label{eqn-2ff_1}
{\big( i\circ \gamma\big)}^{\cdot} (t)=\sum_{k=1}^N\left\langle {\big( i\circ \gamma\big)}^{\cdot},A^k\gamma(t)\right\rangle Y^k(\gamma(t)), \;\; t\in I.
\end{equation}
 Let us note that  by the formula \cite[(2.5)]{Brz+Ondr_2009}, see also \cite[Corollary 4.8]{ONeill_1983}, we have
\begin{equation}\label{eqn-2ff_2}
{\big( i\circ \gamma\big)}^{\cdot\cdot}(t)=\nabla_{\dot\gamma(t)}\dot\gamma(t)+\bold S_{\gamma(t)}(\dot\gamma(t),\dot\gamma(t)), \;\; t\in I.
\end{equation}
In other words, for $t\in I$, the tangential part of ${\big( i\circ \gamma\big)}^{\cdot\cdot}(t)$  is equal to $\nabla_{\dot\gamma(t)}\dot\gamma(t)$ while the normal part of $\ddot\gamma(t)$  is equal to $\mathbf S_{\gamma(t)}(\dot\gamma(t),\dot\gamma(t))$. Since $A^k\gamma(t)\in T_{\gamma(t)}M$ by part  \eqref{ass-M3_2} of Assumption \textbf{M3}, in view of identity \eqref{ass-M4_2} we infer that
\begin{equation}\label{eqn-9.11}
\sum_{k=1}^N\langle {\big( i\circ \gamma\big)}^{\cdot\cdot}(t),A^k\gamma(t)\rangle Y^k(\gamma(t))=\sum_{k=1}^N\langle\nabla_{\dot\gamma(t)}\dot\gamma(t),A^k\gamma(t)\rangle Y^k(\gamma(t))=\nabla_{\dot\gamma(t)}\dot\gamma(t), \; t\in I.
\end{equation}
On the other hand, by taking the standard $\mathbb{R}^n$-valued  derivative of \eqref{eqn-2ff_1}  we get
\begin{eqnarray}
\nonumber
{\big( i\circ \gamma\big)}^{\cdot\cdot}(t)&=&
\sum_{k=1}^N\left\langle {\big( i\circ \gamma\big)}^{\cdot\cdot}(t),A^k\gamma(t)\right\rangle Y^k(\gamma(t))
+\sum_{k=1}^N\left\langle {\big( i\circ \gamma\big)}^{\cdot}(t),A^k {\big( i\circ \gamma\big)}^{\cdot}(t)\right\rangle Y^k(\gamma(t))
\\
\label{eqn-9.12}
&+&\sum_{k=1}^N\left\langle {\big( i\circ \gamma\big)}^{\cdot\cdot}(t),A^k\gamma(t)\right\rangle (d_{\gamma(t)}Y^k)\big({\big( i\circ \gamma\big)}^{\cdot}(t)\big), \;\; t\in I,
\end{eqnarray}
where as before $\lb\cdot,\cdot \rangle$  denotes the standard inner product in $\mathbb{R}^n$.

Finally, we  note that since $A^k$ are skew-symmetric, we have that
\begin{equation}\label{eqn-9.13}
\left\langle  {\big( i\circ \gamma\big)}^{\cdot}(t),A^k {\big( i\circ \gamma\big)}^{\cdot}(t)\right\rangle Y^k(\gamma(t))=0, \; t\in I.
\end{equation}

Hence, from \eqref{eqn-2ff_2} and \eqref{eqn-9.11}, \eqref{eqn-9.12} and \eqref{eqn-9.13} we infer that
\begin{equation}\label{eqn-9.14}
\bold S_{\gamma(t)}( \dot \gamma(t), \dot \gamma(t))=\sum_{k=1}^N\left\langle {\big( i\circ \gamma\big)}^{\cdot}(t),A^k\gamma(t)\right\rangle (d_{\gamma(t)}Y^k) \big({\big( i\circ \gamma\big)}^{\cdot}(t)\big), \;\;
t\in I.
\end{equation}

Putting $t=0$ in the  equality \eqref{eqn-9.14} we get identity \eqref{eqn-2ff}. The proof is now complete.
\end{proof}

\section{Outline of the proof of the main theorem}
\label{sec:outline}

The main idea of the proof of Theorem \ref{thm-main} can be seen from the following result. The proof of this result follows from our new It\^o formula in Lemma \ref{lem-ito} and it uses the material discussed in Section \ref{sec:manifold}. The proof of the converse part can be reproduced from the proof of Lemma \ref{lem-identification}.

\begin{proposition}\label{prop-main} Assume that $M$ is a compact Riemannian homogeneous space and that the coefficients $f$ and $g$ satisfy Assumption  \ref{assum-f+g}. Suppose that a  system
\begin{equation}\label{eqn-701}
\Big(\Omega,\mathscr{F},\mathbb{F},\mathbb{P},W,(u,v)\Big)
\end{equation}
is  a weak solution of \eqref{equa1}. Assume that $A: \mathbb{R}^d\to \mathbb{R}^d$ is a skew-symmetric linear operator satisfying the condition \eqref{ass-M3_2}. Define a process $\mathbf M$ by the following formula
\begin{equation}\label{eqn-702}
\mathbf M(t):=\langle v (t),Au (t)\rangle_{\mathbb{R}^n},\; t \geq 0.
\end{equation}
Then
for every  function $\varphi\in H^1_{\textrm{comp}}$ the following equality holds almost surely
\begin{eqnarray}
\label{eqn-703}
\left\langle\varphi,\mathbf M(t)\right\rangle&=&\left\langle\varphi,\mathbf M(0)\right\rangle-\sum_{k=1}^d\left\langle\partial_{x_k}\varphi,\int_0^t\left\langle\partial_{x_k} u (s), u(s)\right\rangle_{\mathbb{R}^n}\,ds\right\rangle
\\
\nonumber
&+&\left\langle\varphi,\int_0^t\left\langle f( u(s), v(s),\nabla u(s)),A u(s)\right\rangle\,d s\right\rangle
\\
\nonumber
&+&\left\langle\varphi,\int_0^t\left\langle g( u(s), v(s),\nabla u(s)),A u(s)\right\rangle\,dW(s)\right\rangle,\qquad t\ge 0.
\end{eqnarray}
Conversely, assume that a system satisfies all the conditions of Definition  \ref{def-solution} of  a weak solution to equation  \eqref{equa1} but \eqref{sol_weak_2}. Suppose that a finite sequence $(A^i)_{i=1}^N$   of skew symmetric linear operators in  $\mathbb{R}^n$    satisfies conditions \eqref{ass-M3_1},\eqref{ass-M3_2} and \textbf{M4}.  For each $i\in\{1,\cdots,N\}$ define a process $\mathbf M^i$ by the formula \eqref{eqn-702}  with $A=A^i$. Suppose that for every  function $\varphi\in H^1_{\textrm{comp}}$ each $\mathbf M^i$  satisfies equality \eqref{eqn-703} with $A=A^i$.
Then the process $(u,v)$ satisfies the equality \eqref{sol_weak_2}.
\end{proposition}

The first step of the proof of Theorem \ref{thm-main} consists of introducing a penalized and regularized stochastic wave equation (\ref{eqn-apprx_m}-\ref{eqn-apprx_m-i}), i.e.
\begin{eqnarray}
\label{eqn-704}
\partial_{tt}U^m &=&\Delta U^m-m\nabla F(U^m)+f^m(U^m,\nabla_{(t,x)} U^m)+g^m(U^m,\nabla_{(t,x)} U^m)\,dW^m
\\
\label{eqn-705}
&&(U^m(0),\partial_{t}U^m(0)) =\Theta,
\end{eqnarray}
Had we assumed that the coefficients $f$ and $g$ were sufficiently regular, we would have simply put $f^m=f$ and $g^m=g$ above. The existence of a unique global solution $Z^m=\big(U^m,V^m\big)$ to the problem (\ref{eqn-704}-\ref{eqn-705}) is more or less standard.
In Section \ref{sec:2nd-approx} by using uniform energy estimates we will show that the sequence $(Z^m)_{m\in\mathbb{N}}$ is tight on an appropriately chosen Fr\'echet space and by employing Jakubowski's generalization \cite{Jakub_1997} of the Skorokhod embedding Theorem we construct a version $z^m$ of $Z^m$ such that
converges in law to a certain process $z$. In order to prove that   $z$ is a weak solution of (\ref{equa2}) we construct processes \Magenta{$M^i_k$} defined by a formula  analogous to formula \eqref{eqn-701}, see formula \eqref{eqn-104}. We prove that the sequence $\mathbf M^i_k$ is convergent and denote the limit by $\mathbf M^i$. Moreover, we  show that $z$ takes values in the tangent bundle $TM$. The proof is concluded by constructing an appropriate Wiener process, see Lemma \ref{prop-Wiener} and showing,
by employing the argument needed to prove the converse part of Proposition \ref{prop-main} that the process $z$ is indeed a weak solution of (\ref{equa2}).
We remark, that our method of constructing weak solutions to stochastic PDEs
does not employ any martingale representation theorem (and we are not aware  of such results in the Fr\'echet spaces anyway).

\section{Preparation for the proof of the main theorem}
\label{sec:prepa}

\subsection{Approximation of coefficients} Let $(\zeta_m)_{m=1}^\infty$ be the approximation of identity introduced in Section \ref{sec:notation}. Let us assume that $J$ is a continuous vector field on $M$, $h$ a continuous real function on $M$ and $\lambda$ a continuous vector bundle homomorphisms from $TM$ to $TM$. Let $\pi:\mathbb{R}^n\to
\mathscr{L}(\mathbb{R}^n,\mathbb{R}^n)$ be a smooth compactly supported function such that for every $p\in M$,  $\pi(p)$ is the orthogonal projection from $\mathbb{R}^n$ onto $T_pM$. The vector field $J$, the function $h$ and the $\mathcal{L}(\mathbb{R}^n,\mathbb{R}^n)$-valued function $\lambda\circ\pi|_M$ (all defined on $M$) can be extended to continuous compactly supported functions, all  denoted again by the same symbols,   $J:\mathbb{R}^n\to\mathbb{R}^n$, $h:\mathbb{R}^n\to\mathbb{R}$ and $\lambda:\mathbb{R}^n\to\mathscr L(\mathbb{R}^n,\mathbb{R}^n)$. By a standard approximation argument (invoking the  convolution with functions $\zeta_m$) we can find  sequences of $C^\infty$-class functions $J^m:\mathbb{R}^n\to\mathbb{R}^n$, $h^m:\mathbb{R}^n\to\mathbb{R}$, $\lambda^m:\mathbb{R}^n\to\mathscr{L}(\mathbb{R}^n,\mathbb{R}^n)$ supported in a compact set in $\mathbb{R}^n$ such that $J^m\to J$, $h^m\to h$ and $\lambda^m\to\lambda$ uniformly on $\mathbb{R}^n$.

 When we specify the above to our given data: continuous vector fields $f_{d+1}$, $g_{d+1}$ on $M$,  continuous functions $f_0$, $g_0$ on $M$ and  continuous vector bundle homeomorphisms $f_1,\dots,f_d$, $g_1,\dots,g_d$ on $TM$, we can construct the following sequences of approximating smooth functions
\begin{equation}
\label{eqn-f^m}
f^m_0,g^m_0:\mathbb{R}^n\to\mathbb{R},\; f^m_i,g^m_i:\mathbb{R}^n\to\mathscr L(\mathbb{R}^n,\mathbb{R}^n), \; i\in\{1,\cdots,d\}\;  f^m_{d+1},g^m_{d+1}:\mathbb{R}^n\to\mathbb{R}^n,\;\; m\in\mathbb{N}, \\
 \end{equation}
 such that  for some $R_0>0$,
 \begin{equation}
 \label{eqn-supp-aprox}
\bigcup_{m\in\mathbb{N}} \bigcup_{i=0}^{d+1}\big[ \supp(f^m_i) \cup \supp(f^m_i)\big] \subset  B(0,R_0)\subset \mathbb{R}^n,
 \end{equation}
 and  the $L^\infty$ norms of $f^m_{d+1}$ and $g^m_{d+1}$ do not exceed the $L^\infty$ norms of $f_{d+1}$ and $g_{d+1}$ respectively, i.e.
\begin{equation}
\label{eqn-f^m-L^infty}
 \vert f^m_{d+1} \vert_{L^\infty(\mathbb{R}^n,\mathbb{R}^n)} \leq   \vert f_{d+1} \vert_{L^\infty(M,\mathbb{R}^n)}, \;\;
 \vert g^m_{d+1} \vert_{L^\infty(\mathbb{R}^n,\mathbb{R}^n)} \leq   \vert g_{d+1} \vert_{L^\infty(M,\mathbb{R}^n)},\; m\in\mathbb{N}.
 \end{equation}
 and,  uniformly on $\mathbb{R}^n$, as $m\to \infty$,
  \begin{equation}
\label{eqn-conv_f^m_itof_i}
 f^m_i\to f_i \mbox{ and } g^m_i\to g_i, \;\;i\in\{0,\cdots,d+1\}.
 \end{equation}

\subsection{Solutions to an approximated problem}\label{partdefsol} Let the Borel probability measure $\Theta$ on $\mathscr H_{\textrm{loc}}(M)$ be as  in Theorem \ref{thm-main}.  In the next section we will prove that for each  $m\in\mathbb{N}$ there exists  a weak solution  of the following problem
\begin{eqnarray}
\label{eqn-apprx_m}
\partial_{tt}U^m &=&\Delta U^m-m\nabla F(U^m)+f^m(U^m,\nabla_{(t,x)} U^m)+g^m(U^m,\nabla_{(t,x)} U^m)\,dW^m
\\
\label{eqn-apprx_m-i}
&&(U^m(0),\partial_{t}U^m(0)) =\Theta,
\end{eqnarray}
where, compare with \eqref{eqn-f}, the coefficients $f^m$ and $g^m$ are defined by
\begin{eqnarray}\label{eqn-710}
&&
f^m(y,w)=\sum_{i=0}^df^m_i(y)w_i+f^m_{d+1}(y),\;(y,w)\in\mathbb{R}^n\times[\mathbb{R}^n]^{d+1},\\
\label{eqn-711}
&& g^m(y,w)=\sum_{i=0}^dg^m_i(y)w_i+g^m_{d+1}(y), \;(y,w)\in\mathbb{R}^n\times[\mathbb{R}^n]^{d+1}.
\end{eqnarray}
In other words, we will show that   for every $m\in\mathbb{N}$, there exists
\begin{trivlist}
\item[(i)] a complete stochastic basis $(\Omega^m,\mathscr{F}^m,(\mathscr{F}^m_t),\mathbb{P}^m)$,
\item[(ii)]
a spatially homogeneous $(\mathscr{F}^m_t)$-Wiener process $W^m$ with spectral measure $\mu$ and
\item[(iii)]
an $(\mathscr{F}^m_t)$-adapted process $Z^m=(U^m,V^m)$ with weakly continuous paths in $\mathscr H_{\textrm{loc}}$
\end{trivlist}
 such that $\Theta$ is equal to the law of $Z^m(0)$ and for every $t\ge 0$ and $\varphi\in\mathscr D(\mathbb{R}^d,\mathbb{R}^n)$ the following equalities hold almost surely,
\begin{eqnarray}\label{approxsolem}
\langle U^m(t),\varphi\rangle_{\mathbb{R}^n}&=&\langle U^m(0),\varphi\rangle_{\mathbb{R}^n}+\int_0^t\langle V^m(s),\varphi\rangle_{\mathbb{R}^n}\,ds
\\
\langle V^m(t),\varphi\rangle_{\mathbb{R}^n}&=&\langle V^m(0),\varphi\rangle_{\mathbb{R}^n}+\int_0^t\langle -m\nabla F(U^m(s))+f^m(Z^m(s),\nabla U^m(s)),\varphi\rangle_{\mathbb{R}^n}\,ds\nonumber
\\
\label{approxsolem_2}
&+&\int_0^t\langle U^m(s),\Delta\varphi\rangle_{\mathbb{R}^n}\,ds+\int_0^t\left\langle g^m(Z^m(s),\nabla U^m(s))\,dW^m_s,\varphi\right\rangle_{\mathbb{R}^n}
\end{eqnarray}
The processes $Z^m$ need not take values in the tangent bundle $TM$ and since the diffusion nonlinearity is not Lipschitz it  only exists in the  weak  probabilistic sense.

\begin{remark} Let us point out that for each $m\in\mathbb{N}^\ast$, $Z^m(0)$ is   $\mathscr{F}^m_0$-measurable $\mathscr H_{\textrm{loc}}(M)$-valued random variables whose law is equal to $\Theta$. In particular, our initial data satisfy $U^m_0(\omega)\in M$ and $V^m_0(\omega)\in T_{U^m_0(\omega)}M$ a.e. for every $\omega\in\Omega$.
\end{remark}

\section{Tightness of the approximations}\label{sec:2nd-approx}

Lemma \ref{basal}  below constitutes  the first step towards proving Theorem \ref{thm-main2}. In its formulation we use the following generalised family of  energy functions $\mathbf e_{x,T,mF}$, where $x\in\mathbb{R}^n$,  $T>0$, $m\in \mathbb{N}$ and the constant $\mathbf s^2$ was  defined in \eqref{eqn-s^2},     compare with \eqref{eqn-e_T},
\begin{equation}\label{eqn-newenergy}
\mathbf e_{x,T,mF}(t,u,v)=\int_{B(x,T-t)}\left\{\frac 12|\nabla u|^2+\frac 12|u|^2+\frac 12|v|^2+mF(u)+\mathbf{s}^2\right\}\,dy, \;  t\in [0,T],\; (u,v)\in\mathscr H_{\textrm{loc}}.
\end{equation}

\begin{lemma}\label{basal} There exists a weak solution $(\Omega^m,\mathscr{F}^m,(\mathscr{F}^m_t),\mathbb{P}^m,Z^m=(U^m,V^m),W^m)$ to \eqref{eqn-apprx_m}-(\ref{eqn-apprx_m-i}) such that
 \begin{equation}\label{ineq_apriori_01}
\mathbb E^m\,\big[1_A(Z^m(0))\sup_{s\in [0,t]}\, L(\mathbf e_{x,T,mF}(s,Z^m(s)))\big] \leq 4e^{C t}\mathbb E^m\,\left[1_A(Z^m(0)L\left(\mathbf e_{x,T,mF}(0,Z^m(0))\right)\right]
\end{equation}
holds for every  $T\ge 0$, $t\in[0,T]$, $A\in\mathscr B(\mathscr H_{\textrm{loc}})$, $m\in\Bbb N$ whenever $L\in C[0,\infty)\cap C^2(0,\infty)$ is a nondecreasing function such that, for some $c>0$,
\begin{equation}\label{basal-growth}
tL^\prime(t)+\max\,\{0,t^2L^{\prime\prime}(t)\}\leq cL(t),\qquad t>0.
\end{equation}
The constant $C$ depends on $c$, {$c_\mu$} and on $\|f^j\|_{L^\infty}$, $\|g^j\|_{L^\infty}$, $ j=0,1,\cdots, d+1$. \end{lemma}

\begin{proof} This is a direct application of Theorem 5.1 and Theorem 5.2 in \cite{o5}. We use bound  \eqref{eqn-f^m-L^infty}  according to which the $\|f^m_{d+1}\|_{L^\infty}$ norm is bounded by $\|f_{d+1}\|_{L^\infty(M)}$.
\end{proof}

In the following Lemma we use the notions introduced in Appendices \ref{sec:space_L} and  \ref{sec:space_C}.
\begin{lemma}\label{tigh3} Assume that $\newr<\min\,\{2,\frac{d}{d-1}\}$. Then
\begin{trivlist}\item[(1)]  the sequence $\{U^m\}$ is tight on $C_w(\mathbb{R}_+;H^1_{\textrm{loc}})$,
\item[(2)] the sequence $\{V^m\}$ is tight on $\mathbb{L}=L^\infty_{\textrm{loc}}(\mathbb{R}_+;L^2_{\textrm{loc}})$,
 \item[(3)] and, for every $i\in \{1,\cdots, N\}$,  the sequence  ${\langle V^m,A^iU^m\rangle_{\mathbb{R}^d}}$ is tight on $C_w(\mathbb{R}_+;L^{\newr}_{\textrm{loc}})$.
 \end{trivlist}
\end{lemma}

\begin{proof}[Prof of Claim 1]   Let us now take and fix $\varepsilon>0$. In view of   Corollary \ref{cor-dery}
 it is enough to find a sequence $\{a_k\}_{k\in\mathbb{N}}$ such that
\begin{equation}\label{eqn-001}
\mathbb{P}^m\,\left(\bigcup_{k=1}^\infty \Big\{\|U^m\|_{L^\infty((0,k);H^1(B_k))}+\|U^m\|_{C^1([0,k];L^2(B_k))}>a_k\Big\} \right)\leq \varepsilon, \;\; m\in\mathbb{N} .
\end{equation}
Let us introduce the following auxiliary notation.
$$Q_{m,k,\delta}=\{\|Z^m(0)\|_{\mathscr H_{2k}}\leq\delta\},\;\; \delta>0,\;k,m\in\mathbb{N}.$$

From the definition \eqref{eqn-newenergy} of the function $\mathbf e$ we  can find a constant $c>0$ such that
\begin{eqnarray}\label{Q_mkdelta-1}
&&\hspace{-2truecm}\lefteqn{\mathbb{E}^m\Big[ 1_{Q_{m,k,\delta}}\big[\|U^m\|_{L^\infty((0,k);H^1(B_k))}+\|U^m\|_{C^1([0,k];L^2(B_k))}\big]\Big]} \\
&\leq& c\mathbb{E}^m\,\Big[1_{Q_{m,k,\delta}}\sup_{s\in [0,k]}\,L(\mathbf e_{0,2k,m}(s,Z^m(s)))\Big],\;\;\mbox{for all $\delta>0$ and $k,m\in\mathbb{N}$}.
\nonumber
\end{eqnarray}
On the other hand, since the sequence $\{\mathbf{s}_m\}$ in (\ref{eqn-newenergy}) is bounded, by applying Lemma \ref{basal} with function $L(\cdot)=\sqrt{2\cdot}$,
we infer that for   $\delta>0$ and $k\in\mathbb{N}$
  we can find a constant $ C_{k,\delta}>0$ such that
\begin{eqnarray}\label{Q_mkdelta}
 c\mathbb{E}^m\,\Big[1_{Q_{m,k,\delta}}\sup_{s\in [0,k]}\,L(\mathbf e_{0,2k,m}(s,Z^m(s)))\Big]\leq C_{k,\delta},\;\; \mbox{ for every }m\in\mathbb{N}.
\end{eqnarray}
Let us put  $$a_k=C_{k,\delta_k}\varepsilon^{-1} 2^{k+1},\;\; k\in\mathbb{N}.$$
Then by the Tchebyshev inequality we infer from  inequalities  (\ref{Q_mkdelta-1}-\ref{Q_mkdelta}) that
\begin{equation}\label{eqn-002}
\mathbb{P}^m\,\left[ 1_{Q_{m,k,\delta_k}} \Big\{\|U^m\|_{L^\infty((0,k);H^1(B_k))}+\|U^m\|_{C^1([0,k];L^2(B_k))}>a_k\Big\} \right]\leq \varepsilon 2^{-k-1}.
\end{equation}
Since the measure $\Theta$ is Radon, for $k\in \mathbb{N}$ we can find  $\delta_k>0$  such that
$\Theta\big(\{z\in\mathscr H_{\textrm{loc}}:\|z\|_{\mathscr H_{2k}}\ge\delta_k\}\big)<\frac{\varepsilon}{2^{k+1}}$.
Hence, since the law of $Z^m$ under $\mathbb{P}^m$ is equal to $\Theta$ we infer that
\begin{equation}\label{eqn-003}
\mathbb{P}^m(Q_{m,k,\delta_k})>1-\frac{\varepsilon}{2^{k+1}},\;\; k\in\mathbb{N}.
\end{equation}
Summing up,  \eqref{eqn-001}  follows from \eqref{eqn-002}  and \eqref{eqn-003}.
\end{proof}

\begin{proof}[Prof of Claim 2] As far as the sequence $\{V^m\}$ is concerned, let us observe  that by Lemma \ref{basal},
in the same way as in inequalities (\ref{Q_mkdelta-1}-\ref{Q_mkdelta}),
$$
\mathbb{E}^m\Big[1_{Q_{m,k,\delta}}\|V^m\|_{L^\infty((0,k);L^2(B_k))}\Big]
\leq \mathbb{E}^m\Big[1_{Q_{m,k,\delta}}\sup_{s\in [0,k]}L(\mathbf e_{0,2k,m}(s,Z^m(s)))\Big]\leq C_{k,\delta}.
$$
\delr{ Hence  for each $k\in\mathbb{N}$
$$
\mathbb{P}^m\,\left\{\|V^m\|_{L^\infty((0,k);L^2(B_k))}>a_k\right\}\le\varepsilon 2^{-k}
$$}
Arguing as above, but now using Corollary \ref{compero},  we deduce that
the sequence $\{V^m\}$ is tight on $\mathbb{L}=L^\infty_{\textrm{loc}}(\mathbb{R}_+;L^2_{\textrm{loc}})$.
\end{proof}

\begin{proof}[Prof of Claim 3]
 Since the assumptions of Lemma \ref{lem-ito} (with $q=2$, $k=n$, $w=v$ and $Y(y)=A^iy$) are satisfied, by the properties of $F$ and $A^i$ listed in Section \ref{sec:manifold}   the following equality
\begin{eqnarray*}
b(V^m(t),A^iU^m(t),{\varphi})&=&b(V^m(0),A^iU^m(0),{\varphi})-\sum_{k=1}^d\int_0^t b\left(\partial_{x_k}U^m(s),A^iU^m(s),{\partial_{x_k}\varphi}\right)\,ds
\\
&+&\int_0^tb\big(f^m(Z^m(s),\nabla U^m(s)),A^iU^m(s),{\varphi}\big)\,ds
\\
&+&\int_0^tb(g^m(Z^m(s),\nabla U^m(s))\,dW^m,A^iU^m(s),{\varphi})
\end{eqnarray*}
holds almost surely for every $t\ge 0$ and $\varphi\in\mathscr D(\mathbb{R}^d)$. In particular, in view of  Appendix  \ref{sec:space_C}, for every $r>\max\,\{d,2\}$ and  $R>0$, the equality
\begin{eqnarray}
\langle V^m(t),A^iU^m(t)\rangle_{\mathbb{R}^n}&=&\langle V^m(0),A^iU^m(0)\rangle_{\mathbb{R}^n}+\sum_{k=1}^d\partial_{x_k}\left[\int_0^t\left\langle\partial_{x_k}U^m(s),A^iU^m(s)\right\rangle_{\mathbb{R}^n}\,ds\right]\nonumber
\\
&+&\int_0^t\langle f^m(Z^m(s),\nabla U^m(s)),A^iU^m(s)\rangle_{\mathbb{R}^n}\,ds\nonumber
\\
&+&\int_0^t\langle g^m(Z^m(s),\nabla U^m(s)),A^iU^m(s)\rangle_{\mathbb{R}^n}\,dW^m\label{reok}
\end{eqnarray}
holds in $\mathbb{W}^{-1,\newr}_R$  for every $t\ge 0$,  almost surely.

Indeed, by the Gagliardo-Nirenberg inequality (G-NI) and the
  H\"older inequality we get
 \begin{equation}\label{eqn-L^rstart}
\|ab\|_{L^{\newr}(\mathbb{R}^d)}\leq \|a\|_{L^2(\mathbb{R}^d)}\|b\|_{L^{2\newr/(2-\newr)}(\mathbb{R}^d)}\leq c\|a\|_{L^2(\mathbb{R}^d)}\|b\|_{H^1(\mathbb{R}^d)},\;\; a\in L^2(\mathbb{R}^d),\, b\in H^1(\mathbb{R}^d).
 \end{equation}
 Therefore,  the first deterministic integral in (\ref{reok}) converges in $L^{\newr}_{\textrm{loc}}$. But  the  the mapping $\partial_{x_k}:L^{\newr}_{\textrm{loc}}\to\mathbb{W}^{-1,\newr}_R$  is continuous for every $R>0$, and so the first term in (\ref{reok}) is a well defined $\mathbb{W}^{-1,\newr}_R$-valued random variable for each $R>0$.

 Since,  the functions $f^m_i$, $g^m_i$ are compactly supported, we can find $T>0$
and $c>0$ such that
\begin{equation}\label{zaklomez}
|\langle f^m(y,w),A^iy\rangle|+|\langle g^m(y,w),A^iy\rangle|\le c\mathbf 1_{[|y|\le T]}(1+|w|),\quad (y,w)\in\mathbb{R}^n\times[\mathbb{R}^n]^{d+1}.
\end{equation}
Therefore, the stochastic and the second deterministic integrals are convergent  in $L^2_{\textrm{loc}}$. Since $r>d$, by the G-NI
\begin{equation}\label{eqn-004}
L^2_{\textrm{loc}} \embed \mathbb{W}^{-1,\newr}_R,\;\; R>0.
 \end{equation}
 Thus  the stochastic and the second deterministic integrals  are  convergent in $\mathbb{W}^{-1,\newr}_R$, $R>0$.

Next let us choose  $p>4$ and $\gamma>0$  such that  $$\gamma+\frac2p<\frac12.$$
\delr{ Let us  denote $$Q_{m,k,\delta}=\{\|Z^m(0)\|_{\mathscr H_{2k}}\le\delta\}.$$}
Let us  denote
\begin{eqnarray}\label{eqn-I^2}
I^{(2)}(t)&=&\int_0^t\left\langle\partial_{x_l}U^m(s),A^iU^m(s)\right\rangle_{\mathbb{R}^n}\,ds,\;\; t\geq 0.
\\
\label{eqn-I^3}
I^{(3)}(t)&=&\int_0^t\langle f^m(Z^m(s),\nabla U^m(s)),A^iU^m(s)\rangle_{\mathbb{R}^n}\,ds, \;\;t\geq 0,
\\
\label{eqn-I^4}
I^{(4)}(t)&=&\int_0^t\langle g^m(Z^m(s),\nabla U^m(s)),A^iU^m(s)\rangle_{\mathbb{R}^n}\,dW^m, \;\;t\geq 0.
\end{eqnarray}
Because of \eqref{eqn-004}, by the Garsia-Rumsey-Roedemich Lemma \cite{Garsia+R+R_1970} and the  the Burkholder inequality we have
\begin{eqnarray}\label{eqn-005}
&&\hspace{-2truecm}\lefteqn{\mathbb E^m\,\Big[ 1_{Q_{m,k,\delta}}\|I^{(4)}\|^p_{C^\gamma([0,k];\mathbb W^{-1,\newr}_k)}\Big]
\leq c_{k,r}\mathbb E^m\,\|1_{Q_{m,k,\delta}}I^{(4)}\|^p_{C^\gamma([0,k],L^2(B_k))}}\\
&\leq& \tilde c_{k,r,p}\mathbb E^m \Big[1_{Q_{m,k,\delta}}\,\int_0^k\|\langle g^m(Z^m(s),\nabla U^m(s)),A^iU^m(s)\rangle_{\mathbb{R}^n}\|^p_{\mathscr T_2(H_\mu,L^2(B_k))}\,ds\Big].
\nonumber
\end{eqnarray}
Applying Lemma \ref{lem-hsop},  inequality (\ref{zaklomez}) and Lemma \ref{basal} we infer that the RHS of the above inequality  is bounded by
\begin{eqnarray*}
&\leq& \bar c_{k,r,p}\mathbb E^m\Big[ 1_{Q_{m,k,\delta}}\,\int_0^k\|\langle g^m(Z^m(s),\nabla U^m(s)),A^iU^m(s)\rangle_{\mathbb{R}^n}\|^p_{L^2(B_k)}\,ds\Big]
\\
&\leq& c^0_{k,r,p}\mathbb E^m \Big[ 1_{Q_{m,k,\delta}}\,\int_0^k \big[1+\|Z^m(s)\|^p_{\mathscr H_k}\big]\,ds\Big]
\\
&\leq& c^0_{k,r,p}\mathbb E^m\Big[ 1_{Q_{m,k,\delta}}\,\int_0^k\big[1+\mathbf e^\frac p2_{0,2k,m}(s,Z^m(s))\big]\,ds\leq C_{k,r,p,\delta}\Big].
\end{eqnarray*}
Summing up, we proved that
\begin{eqnarray}\label{eqn-006}
\mathbb E^m\,\Big[ 1_{Q_{m,k,\delta}} \|I^{(4)}\|^p_{C^\gamma([0,k];\mathbb W^{-1,\newr}_k)}\Big] &\leq& C_{k,r,p,\delta}.
\end{eqnarray}
Analogously, by the H\"older inequality and Lemma \ref{basal},
\begin{eqnarray}\nonumber
&&\hspace{-2truecm}\lefteqn{\mathbb E^m\,\Big[ 1_{Q_{m,k,\delta}} \|I^{(3)}\|^p_{C^\gamma([0,k];\mathbb W^{-1,r^\prime}_k)}\Big]
\leq c_{k,r}\mathbb E^m\,\|1_{Q_{m,k,\delta}}I^{(3)}\|^p_{C^\gamma([0,k],L^2(B_k))}}\\
\nonumber
&\leq& \tilde c_{k,r,p}\mathbb E^m \Big[ 1_{Q_{m,k,\delta}}\,\int_0^k\|\langle f^m(Z^m(s),\nabla U^m(s)),A^iU^m(s)\rangle_{\mathbb{R}^n}\|^p_{L^2(B_k)}\,ds\Big]
\\
&\leq& c^0_{k,r,p}\mathbb E^m \Big[ 1_{Q_{m,k,\delta}}\,\int_0^k\big[1+\|Z^m(s)\|^p_{\mathscr H_k}\big]\,ds\Big]
\nonumber
\\
\label{eqn-007}
&\leq& c^0_{k,r,p}\mathbb E^m \Big[1_{Q_{m,k,\delta}}\,\int_0^k\big[1+\mathbf e^\frac p2_{0,2k,m}(s,Z^m(s))\big]\,ds\Big] \leq C_{k,r,p,\delta}.
\end{eqnarray}

Concerning the process $I^{(2)}$, the H\"older inequality yields
\begin{eqnarray*}
\Vert\partial_{x_l}I^{(2)}\Vert^p_{C^\gamma([0,k],\mathbb{W}^{-1,\newr}_k)}& \leq& \Vert I^{(2)}\Vert^p_{C^\gamma([0,k],L^{\newr}(B_k))}
\\
&\leq& c_{k,p}\int_0^k\Vert\left\langle\partial_{x_l}U^m(s),A^iU^m(s)\right\rangle_{\mathbb{R}^n}\Vert^p_{L^{\newr}(B_k)}\,ds
\\
&\leq& \tilde c_{k,r,p}\int_0^k\|\nabla U^m(s)\|^p_{L^2(B_k)}\|U^m(s)\|^p_{H^1(B_k)}\,ds\\
&\leq& \tilde c_{k,r,p}\int_0^k\mathbf e^p_{0,2k,m}(s,Z^m(s))\,ds.
\end{eqnarray*}
Hence, by Lemma \ref{basal} we infer that
\begin{eqnarray}
\label{eqn-008}
&&\mathbb E^m\,\Big[ 1_{Q_{m,k,\delta}}\Vert\partial_{x_l}I^{(2)}_m\Vert^p_{C^\gamma([0,k],\mathbb{W}^{-1,\newr}_k)}\Big] \leq C_{k,p,\delta},
\end{eqnarray}

Moreover, we have
\begin{eqnarray*}
&&\hspace{-2truecm}\lefteqn{\mathbb{E}^m \Big[1_{Q_{m,k,\delta}}\sup_{s\in [0,k]}\|\langle V^m(s),A^iU^m(s)\rangle_{\mathbb{R}^n}\|^p_{\mathbb{W}^{-1,\newr}_k}\Big]}
\\
&\leq &\mathbb{E}^m \Big[1_{Q_{m,k,\delta}}\sup_{s\in [0,k]}\|\langle V^m(s),A^iU^m(s)\rangle_{\mathbb{R}^n}\|^p_{L^{\newr}(B_k)}\Big]
\\
&\leq & c_r\mathbb{E}^m\Big[1_{Q_{m,k,\delta}}\sup_{s\in [0,k]}\|V^m(s)\|^p_{L^2(B_k)}\|U^m(s)\|^p_{H^1(B_k)}\Big]
\\
&\leq &\tilde c_r\mathbb{E}^m\Big[ 1_{Q_{m,k,\delta}}\sup_{s\in [0,k]}\,\mathbf e^p_{0,2k,m}(s,Z^m(s))\Big]\leq c_{k,r,p,\delta}.
\end{eqnarray*}
Summing up, we proved that there exists a constant $C_{k,r,p,\delta}>0$ such that
\begin{eqnarray}
\label{eqn-009}
&&\hspace{-2truecm}
\lefteqn{
\mathbb E^m\,\Big\{ 1_{Q_{m,k,\delta}}\Big[ \Vert\langle V^m,A^iU^m\rangle_{\mathbb{R}^n}\Vert_{L^\infty([0,k];L^{\newr}(B_k))}
}
\\
&+&\Vert\langle V^m,A^iU^m\rangle_{\mathbb{R}^n}\Vert_{C^\gamma([0,k],\mathbb{W}^{-1,\newr}_k)}\Big]^p \Big\}\leq C_{k,\newr,p,\delta}.
\nonumber
\end{eqnarray}

Hence, by Proposition \ref{cor-dery}  we conclude that the sequence $\langle V^m,A^iU^m\rangle_{\mathbb{R}^n}$ is tight on $C_w(\mathbb{R}_+;L^{\newr}_{\textrm{loc}})$. Indeed, for any fixed  $\eps>0$  we can find a sequence
$\big(a_k\big)_{k=1}^\infty $ of positive real numbers such that $$C_{k,r,p,\delta_k}=a^p_k \eps 2^{-k-1}, \;\; k\in\mathbb{N}.$$
Then, by \eqref{eqn-009} and \eqref{eqn-003}, we infer that for each $k\in\mathbb{N}$,
\begin{equation}
\label{eqn-010}
\mathbb P^m\,\left\{\Vert\langle V^m,A^iU^m\rangle_{\mathbb{R}^n}\Vert_{L^\infty([0,k];L^{\newr}(B_k))}+\Vert\langle V^m,A^iU^m\rangle_{\mathbb{R}^n}\Vert_{C^\gamma([0,k],\mathbb{W}^{-1,\newr}_k)}>a_k\right\}\leq\varepsilon 2^{-k}.
\end{equation}
Hence, Proposition \ref{cor-dery} is applicable.

\end{proof}

\section{Proof of the main result}\label{sec:proof_main}

Let us consider the approximating sequence of processes $\big(Z^m\big)_{m\in\mathbb{N}}=\big((U^m,V^m)\big)_{m\in\mathbb{N}}$ from Lemma \ref{basal} and consider also  the following representation of Wiener processes $W^m$:
\begin{equation}\label{devel}
W_t^m=\sum_i\beta^m_i(t)e_i,\qquad t\geq 0,
\end{equation}
where $\beta=(\beta^1,\beta^2,\dots)$ are independent real standard Wiener processes and $\{e_i:i\in\mathbb{N}\}$ is an orthonormal basis in $H_\mu$, see Proposition \ref{prop_H_mu}.

Assume that $r>\max\,\{d,2\}$ is fixed. Then Lemma \ref{tigh3}, Corollary \ref{coli}, Proposition \ref{cacoli} and Corollary \ref{fiza} yield\footnote{Let us recall that we used there to denote by $\mathbb{L}$
the space $L^\infty_{\textrm{loc}}(\mathbb{R}_+;L^2_{\textrm{loc}})$} that there exists
\begin{itemize}
\item a probability space $(\Omega,\mathscr{F},\mathbb{P})$,
\item a subsequence $m_k$,
\item the following sequences of Borel measurable functions
\begin{equation}\label{eqn-102}
\mbox{
\begin{tabular}{|c| c |c|}
  \hline
$(u^k)_{k\in\mathbb{N}}$   & with values in  &  $C(\mathbb{R}_+,H^1_{\textrm{loc}})$
   \\\hline
$(v^k)_{k\in\mathbb{N}}$& \bysame &  $C(\mathbb{R}_+,L^2_{\textrm{loc}})$ \\\hline
 $(w^k)_{k\in\mathbb{N}}$  &\bysame&  $C(\mathbb{R}_+,\mathbb{R}^{\mathbb{N}})$ \\\hline
\end{tabular}
}\end{equation}

\item the following Borel random variables
\begin{equation}\label{eqn-103}\mbox{
\begin{tabular}{|c| c |c|}
  \hline
$v_0$  & with values in  &   $L^2_{\textrm{loc}}$
   \\\hline
$u$ &\bysame &  $C_w(\mathbb{R}_+;H^1_{\textrm{loc}})$  \\\hline
$\bar v$ &\bysame &  $L^\infty_{\textrm{loc}}(\mathbb{R}_+;L^2_{\textrm{loc}})$  \\\hline
$w$ &\bysame &  $C(\mathbb{R}_+,\mathbb{R}^{\mathbb{N}})$  \\\hline
$M^i$, $i=1,\cdots, N$ &\bysame &   $C_w(\mathbb{R}_+;L^{\newr}_{\textrm{loc}})$  \\\hline
\end{tabular}
}\end{equation}
\end{itemize}
such that, with the notation $z^k=(u^k,v^k)$, $k\in\mathbb{N}$ and
\begin{equation}\label{eqn-104}
M^i_k:=\langle v^k,A^iu^k\rangle_{\mathbb{R}^n},\;i=1,\cdots,N,\;k\in\mathbb{N},
\end{equation}
i.e. $M^i_k(t,\omega,x):=\langle v^k(t,\omega,x),A^iu^k(t,\omega,x)\rangle_{\mathbb{R}^n}$,
the following conditions are satisfied.
\begin{itemize}
\item[(R1)] for every $k\in\mathbb{N}$, the law of $(Z^{m_k},\beta^{m_k})$ coincides with the law of $(z^k,w^k)$ on
 $\mathscr B(C(\mathbb{R}_+,\mathscr H_{\textrm{loc}})\times C(\mathbb{R}_+,\mathbb{R}^{\mathbb{N}}))$,
\item[(R2)]  pointwise  on $\Omega$ the following  convergences hold

\begin{equation}\label{eqn-105}\mbox{
\begin{tabular}{|c| c |c|}
  \hline
    &&\\
  $u^k\to u$  & in & $C_w(\mathbb{R}_+;H^1_{\textrm{loc}})$ \\
    &&\\
    \hline
    &&\\
  $v^k\to \bar v$  & \bysame  & $L^\infty_{\textrm{loc}}(\mathbb{R}_+;L^2_{\textrm{\textrm{loc}}})$ \\
    &&\\
    \hline
    &&\\
  $v^k(0)\to v_0$ & \bysame  & $L^2_{\textrm{loc}}$ \\
    &&\\
    \hline
    &&\\
 $M^i_k \to M^i$ & \bysame  & $C_w(\mathbb{R}_+;L^{\newr}_{\textrm{loc}})$ \\
    &&\\
    \hline
    &&\\
 $w^k\to w$ & \bysame  & $C(\mathbb{R}_+,\mathbb{R}^{\mathbb{N}})$ \\
    &&\\
  \hline
\end{tabular}
}\end{equation}
 \item[(R3)]  the law of $(u(0),v_0)$ is equal to $\Theta$.
\end{itemize}
In particular, the conclusions of Lemma \ref{basal} hold for this new system of processes. This is summarized in the Proposition below.

\begin{proposition}\label{prop-basal2} If  $\rho$ is the  constant from  Lemma \ref{basal}, then the  inequality \eqref{ineq_apriori_01} holds. Thus, for any nondecreasing function $L\in C[0,\infty)\cap C^2(0,\infty)$ satisfying the condition (\ref{basal-growth}), we have
\begin{equation}\label{integ}
\mathbb{E}\,\left[\mathbf 1_A(z^k(0))\sup_{s\in[0, t]}L(\mathbf e_{x,T,m_k}(s,z^k(s)))\right]\leq 4e^{\rho t}\mathbb{E}\,\left[\mathbf 1_A(z^k(0))L\left(\mathbf e_{x,T,m_k}(0,z^k(0))\right)\right]
\end{equation}
 for every $k\in\mathbb{N}$, $t\in [0,T]$, $x\in\mathbb{R}^d$, $A\in\mathscr B(\mathscr H_{\textrm{loc}})$.
\end{proposition}
Before we continue, let us observe that compactness of $H^1_{\textrm{loc}} \embed L^2_{\textrm{loc}}$ and the properties  \eqref{eqn-103} and \eqref{eqn-105} imply the following auxiliary result.
\begin{proposition}\label{prop-102}
In the above framework, all the trajectories of the process  $u$ belong to
$C(\mathbb{R}_+,L^2_{\textrm{loc}})$ and for every $t\in \mathbb{R}_+$, $u^k(t)\to u(t)$ in $L^2_{\textrm{loc}}$.
\end{proposition}

We  also introduce a  filtration $\mathbb{F}=\big(\mathscr{F}_t\big)_{t\geq 0}$ of $\sigma$-algebras on the probability space $(\Omega,\mathscr{F},\mathbb{P})$ defined by
\begin{equation}\label{eqn-sigma-field}
\mathscr{F}_t=\sigma\{\sigma\{v_0,u(s),w(s):s\in [0,t]\}\cup\{N:\mathbb{P}\,(N)=0\}\}, \;\;\; t\geq 0.
\end{equation}

Our first result states,  roughly speaking, that the limiting process $u$ takes values in the set $M$. To be precise, we have the following.

\begin{proposition}\label{prop-set-Q_u}
There exists a set $Q_u\in\mathscr{F}$ such that $\mathbb{P}\,(Q_u)=1$ and, for every $\omega\in Q_u$ and $t\ge 0$,  $u(t,\omega)\in M$ almost
everywhere on $\mathbb{R}^d$.
\end{proposition}
\begin{proof} Let us fix $T>0$ and $\delta >0$. In view of the definition \eqref{eqn-newenergy} of the function $\mathbf e_{0,T,m}$, the inequality (\ref{integ})
yields that for some finite constant $ C_{T,\delta}$,
\begin{equation}\label{eqn-108}
\mathbb{E}\,\Big[1_{\big\{\|z^k(0)\|_{\mathscr H_{T}}\leq \delta\big\}}\int_{B_{T-t}}m_kF(u^k(t))\,dx\Big]\leq C_{T,\delta},\qquad t\in[0,T].
\end{equation}

Since $\|z^k(0)\|_{\mathscr H_T}\to\|(u(0),v_0)\|_{\mathscr H_T}$ and by Proposition \ref{prop-102} for every $t\in [0,T]$,  $u^k(t)\to u(t)$ in $L^2(B_{T-t})$, by applying the Fatou Lemma we infer that
\begin{equation}\label{eqn-109}
\mathbb{E}\,\Big[1_{\big\{\|(u(0),v_0)\|_{\mathscr H_{T}}<\delta\big\}}\int_{B_{T-t}}F(u(t))\,dx\Big]\leq \liminf_{k\to\infty}
\mathbb{E}\,\Big[1_{\big\{\|z^k(0)\|_{\mathscr H_{T}}\leq \delta\big\}}\int_{B_{T-t}}F(u^k(t))\,dx\Big].
\end{equation}
On the other hand, by \eqref{eqn-108}, since $m_k\toup\infty$,

\begin{equation}\label{eqn-110}
\liminf_{k\to\infty}
\mathbb{E}\,\Big[1_{\big\{\|z^k(0)\|_{\mathscr H_{T}}\leq \delta\big\}}\int_{B_{T-t}}F(u^k(t))\,dx\Big]=0
\end{equation}
and so
\begin{equation}\label{eqn-111}
\mathbb{E}\,\Big[1_{\big\{\|(u(0),v_0)\|_{\mathscr H_{T}}<\delta\big\}}\int_{B_{T-t}}F(u(t))\,dx\Big]=0.
\end{equation}

 Taking  the limits as $\delta\toup \infty$ and $T\toup \infty$ we get that
\begin{equation}\label{eqn-F=0}
\mathbb{E}\,\int_{\mathbb{R}^d}F(u(t))\,dx=0,\qquad t\ge 0.
\end{equation}
Since $F\geq 0$ and   $M=F^{-1}(\{0\})$ we infer that for each $t\in\mathbb{R}_+$, $u(t,x)\in M$ for $\mathrm{Leb}$ a.a. $x\in \mathbb{R}^d$, $\mathbb{P}$-almost surely.  Hence there exists a set $\Omega_u$ of full measure such that
$$ \mathrm{Leb} \big(\{x \in \mathbb{R}^d:u(q,\omega,x)\notin M\}\big) =0 \mbox{ for every  } q\in \mathbb{Q}_+ \mbox{ and every }\omega\in\Omega_u.$$

Let us take $t\in\mathbb{R}_+$,  $\omega\in\Omega_\ast$. Obviously  we can find a sequence $(q_n)_{n=1}^\infty\subset \mathbb{Q}_+$ such that  $q_n\to t$. Hence, by Proposition \ref{prop-102}  there exists a subsequence $(n_k)_{k=1}^\infty$ such that for $\mathrm{Leb}$-almost every $x$, $u(q_{n_k},\omega,x)\to u(t,\omega,x)$ as $k\to \infty$.
Hence by closedness of the set $M$ we infer that also
$u(t,\omega,x)\in M$ for $\mathrm{Leb}$-almost every $x$. The proof is complete.
\end{proof}

The last result  suggests the following definition.
\begin{definition}\label{def-bfu} Set
\begin{equation}\label{eqn-bfu}
\mathbf{u}(t,\omega)=\begin{cases} u(t,\omega), & \mbox{ for } t\ge 0 \mbox{ and }\omega\in Q_u,\cr
 \mathbf{p}, & \mbox{ for } t\ge 0 \mbox{ and } \omega\in\Omega\setminus Q_u.
 \end{cases}
\end{equation}
 where $\mathbf{p}(x)=p$, $x\in\mathbb{R}^d$ for some fixed (but otherwise arbitrary) point  $p\in M$.
\end{definition}

Let $\bar v$ be the $\mathbb{L}$-valued random variable as in \eqref{eqn-103} and \eqref{eqn-105}. In view of Proposition \ref{prop-wis} there exits a a measurable $L^2_{\textrm{loc}}$-valued process $v$  such that for every $\omega\in\Omega$, the function $v(\cdot,\omega)$ is a
representative of $\bar v(\omega)$.

\begin{lemma}\label{thick-V}
There exists an $\mathbb{F}$-progressively measurable $L^2_{\textrm{loc}}$-valued process $\mathbf V$  such that $\Leb\otimes\mathbb{P}$-a.e., $\mathbf V=v$   and,  $\mathbb{P}$-almost surely,
\begin{equation}\label{feq}
\mathbf{u}(t)=\mathbf u(0)+\int_0^t\mathbf V(s)\,ds, \; \mbox{ in } L^2_{\textrm{loc}},\; \mbox{ for all } t\ge 0.
\end{equation}
 Moreover $\mathbf V(t,\omega)\in T_{\mathbf u(t,\omega)}M$,  $\mathrm{Leb}$-a.e. for every $(t,\omega)\in\mathbb{R}_+\times\Omega$. Finally, there exists an $\mathscr{F}_0$-measurable $L^2_{\textrm{loc}}$-valued random variable $\mathbf v_0$  such that
 \begin{equation}\label{eqn-115}
\mathbf v_0=v_0,  \;\;\mathbb{P} \mbox{ almost surely}
\end{equation}
 and,
  for every $\omega\in\Omega$,
   \begin{equation}\label{eqn-116}
\mathbf v_0(\omega)\in T_{\mathbf u(0,\omega)}M,  \mathrm{Leb}-\mbox{a.e.}.
 \end{equation}
\end{lemma}
\begin{proof}
Let us fix $t>0$. Since the map
$$
C(\mathbb{R}_+,H^1_{\textrm{loc}})\times C(\mathbb{R}_+,L^2_{\textrm{loc}}) \ni (u,v)\mapsto u(t)-u(0)-\int_0^tv(s)\,ds \in  L^2_{\textrm{loc}}
$$
is continuous,  by identity \eqref{approxsolem} and (R1) we infer that
$$
u^k(t)=u^k(0)+\int_0^tv^k(s)\,ds,\qquad t\ge 0
$$
almost surely. Here we used the following simple rule. If $R(X)=0$ a.s., $R$ is a Borel mapping and $X$ has the same law as $Y$ then  $R(Y)=0$ a.s. Hence, if $\varphi\in L^2_{\textrm{comp}}$
then $\mathbb{P}$-almost surely
\begin{eqnarray}\label{eqn-117}
&&\hspace{-2truecm}\lefteqn{
\langle\varphi,u(t)\rangle-\langle\varphi,u(0)\rangle-\int_0^t\langle\varphi,v(s)\rangle\,ds
}\\
&=&\lim_{k\to\infty}\left[\langle\varphi,u^k(t)\rangle-\langle\varphi,u^k(0)\rangle-\int_0^t\langle\varphi,v^k(s)\rangle\,ds\right]=0.
\nonumber
\end{eqnarray}
Let us define and $L^2_{\textrm{loc}}$-valued process $q$ by
$$
q(t,\omega)=\begin{cases} \lim_{j\to\infty}j\big[u(t,\omega)-u\big((t-\frac1j )^+,\omega\big)\big], & \mbox{if the $L^2_{\textrm{loc}}$-limit exists}
\cr
0& \mbox{ otherwise,}
\end{cases}
$$
where for $x\in\mathbb{R}$, by $x^+$ we denote the positive part of $x$.

  Then $q$ is an $L^2_{\textrm{loc}}$-valued $\mathbb{F}$-progressively measurable  and by \eqref{eqn-117},
  $$q=v \,\,\mathrm{Leb}\otimes\bar{\mathbb{P}}\mbox{-almost everywhere.}$$
  In particular, there exists a $\mathbb{P}$-conegligible  set $N\subset \Omega$ such that $q(\cdot,\omega)=v(\cdot,\omega)$ a.e. for every $\omega\in N$. Hence all  the paths of the process $1_{N}q$ belong to the space $L^\infty_{\textrm{loc}}(\mathbb{R}_+;L^2_{\textrm{loc}})$. Hence \eqref{feq} follows provided we define the process $\mathbf{V}$ to be equal $1_{N}q$.

Concerning the 2nd and the 3rd issue, let us observe that  by   \cite{Ham_1975} p. 108, there exists a smooth compactly supported function $H:\mathbb{R}^n\to\mathbb{R}^n$ such that, for every $p\in M$, $H(p)=p$ and

\begin{eqnarray}\label{eqn-Ham}
H^\prime(p)\xi&=&\xi \iff  \xi\in T_pM,
\end{eqnarray}
where $H^\prime(p)=d_pH\in \mathcal{L}(\mathbb{R}^n,\mathbb{R}^n)$ is the Fr\'echet derivative of $H$ at $p$.
Since by \eqref{eqn-104}, $\mathbb{P}$-almost surely, the following identity is satisfied in $L^2_{\textrm{loc}}$, for every $ t\ge 0$,
$$
\int_0^tH^\prime(\mathbf u(s))v(s)\,ds=H(\mathbf u(t))-H(\mathbf u(0))=\mathbf u(t)-\mathbf u(0)=\int_0^tv(s)\,ds,
$$
  we may conclude that $v=H^\prime(\mathbf u)v$ for $\mathrm{Leb}\otimes\mathbb{P}$-almost every $(t,\omega)$.
   Hence,  $\mathbf{V}=H^\prime(\mathbf u)(\mathbf{V})$ on a $\mathbb{F}$-progressively measurable and $\mathrm{Leb}\otimes\mathbb{P}$-conegligible set. This, in view of \eqref{eqn-Ham}, implies equality \eqref{eqn-116}.

 Finally, in order to prove \eqref{eqn-103}, let us observe that  since the following map  $$L^2_{\textrm{loc}}\times L^2_{\textrm{loc}}\to L^2_{\textrm{loc}}:(u,v)\mapsto H^\prime(u)v-v$$ is continuous,
$$
H^\prime(u^k(0))v^k(0)=v^k(0), \;\mbox{ for every } k\in\mathbb{N}, \;\mbox{almost surely}.
$$
 Therefore, almost surely,
$$
H^\prime(\mathbf u(0))v_0=v_0.
$$
\end{proof}

\begin{lemma}\label{verin}
There exists a $\mathbb{P}$-conegligible set $Q\in\mathscr{F}$ such that the following properties are satisfied.
\begin{trivlist}\item[(i)]
For every $i\in\{1,\cdots,N\}$ the process $\mathbf M^i$ defined by
\begin{equation}\label{eqn-M}
\mathbf M^i=1_QM^i
\end{equation}
 is $L^2_{\textrm{loc}}$-valued $\mathbb{F}$-adapted, with  weakly continuous paths.
 \item[(ii)]
   The following three identities hold for every $\omega\in Q$,
\begin{eqnarray}\label{eqn-M^i=VA^i}
\mathbf M^i(t,\omega)&=&\left\langle\mathbf V(t,\omega),A^i\mathbf u(t,\omega)\right\rangle_{\mathbb{R}^n}\quad \mbox{ for a.e. }t\geq 0
\\
\label{eqn-v_0}
\mathbf v_0(\omega)&=&\sum_{i,j=1}^Nh_{ij}(\mathbf u(0,\omega))\mathbf M^i(0,\omega)A^j\mathbf u(0,\omega)
\\
\label{eqn-v=V}
\mathbf V(t,\omega)&=&\sum_{i,j=1}^Nh_{ij}(\mathbf u(t,\omega))\mathbf M^i(t,\omega)A^j\mathbf u(t,\omega), \;\; \mbox{ for a.e. } t\geq 0
\end{eqnarray}
\end{trivlist}
Moreover, if the process $\mathbf{v}$ is defined by
\begin{equation}
\label{vsolu}
\mathbf{v}(t,\omega):=\sum_{i,j=1}^Nh_{ij}(\mathbf u(t,\omega))\mathbf M^i(t,\omega)A^j\mathbf u(t,\omega), \;\; \omega\in \Omega, t\geq 0,
\end{equation}
then for every $\omega\in Q$,
\begin{eqnarray}
\label{vsolu3}
\mathbf{v}(t,\omega)&=&\mathbf V(t,\omega), \mbox{ for a.e. } t\geq 0,\\
\label{vsolu2}
\mathbf{v}(t,\omega)&\in&T_{\mathbf u(t,\omega)M},\qquad t\ge 0.
\end{eqnarray}

Finally, with $\mathbf{z}=(\mathbf{u},\mathbf{v})$, for every $\omega\in Q$,   for almost every $t\ge 0$,
\begin{eqnarray}
\label{reinf}
\lim_{k\to\infty}\left\langle f^{m_k}(z^k(t,\omega)),A^iu^k(t,\omega)\right\rangle_{\mathbb{R}^n}  = \left\langle f(\mathbf{z}(t,\omega),\nabla \mathbf{u}(t,\omega)),A^i\mathbf{u}(t,\omega)\right\rangle_{\mathbb{R}^n},
\\
\label{reing}
\lim_{k\to\infty}\left\langle g^{m_k}(z^k(t,\omega)),A^iu^k(t,\omega)\right\rangle_{\mathbb{R}^n}  = \left\langle g(\mathbf{z}(t,\omega),\nabla \mathbf{u}(t,\omega)),A^i\mathbf{u}(t,\omega)\right\rangle_{\mathbb{R}^n},
\end{eqnarray}
where the limits are with respect to the weak topology on  $L^2_{\textrm{loc}}$.
\end{lemma}
\begin{proof}
Since by \eqref{eqn-105} $v^k\to {\bar v}$ in $L^\infty_{\textrm{loc}}(\mathbb{R}_+;L^2_{\textrm{loc}})$ on $\Omega$, we infer that for every $R>0$ and $\omega\in\Omega$ the sequence
$\|v^k(\omega)\|_{L^\infty((0,R),L^2(B_R))}$, $ k\in \mathbb{N}$
is  bounded. Since by \eqref{eqn-102}, $v^k$ is a continuous $L^2_{\rm{loc}}$-valued process, $\|v^k(\omega)\|_{L^\infty((0,R),L^2(B_R))}=\|v^k(\omega)\|_{C([0,R],L^2(B_R))}$
and hence  also the following sequence
\begin{equation}\label{eqn-10.28}
\|v^k(\omega)\|_{C([0,R],L^2(B_R))}, \;\; k\in \mathbb{N}
\end{equation}
is  bounded.

Let us now fix $t\ge 0$ and $\omega\in\Omega$. Then by  \eqref{eqn-10.28}
for every $R>0$ the sequence $\|v^k(t,\omega)\|_{L^2(B_R)}$ is bounded. Hence, by employing the diagonalization procedure,  we can find an element  $\theta(t,\omega)\in L^2_{\textrm{loc}}$ and a subsequence $(k_j)_{j}$, depending on $t$ and $\omega$, such that \begin{equation}\label{eqn-10.29}v^{k_j}(t,\omega)\to\theta(t,\omega) \mbox{ weakly in }L^2_{\textrm{loc}}.
\end{equation}
Since by Proposition \ref{prop-102}
$u^{k_j}(t,\omega)\to u(t,\omega)$  strongly in $L^2_{\textrm{loc}}$,
  we infer that $\langle A^iu^{k_j}(t,\omega),v^{k_j}(t,\omega)\rangle_{\mathbb{R}^n}$ converges to $\langle A^iu(t,\omega),\theta(t,\omega)\rangle_{\mathbb{R}^n}$ in the sense of distributions and hence
 for any $\varphi\in\mathscr D$
\begin{eqnarray}\label{eqn-10.31}
\langle M^i(t,\omega),\varphi\rangle&=&\lim_{j\to\infty} \langle M^i_{k_j}(t,\omega),\varphi\rangle
=\lim_{j\to\infty}\langle\langle A^iu^{k_j}(t,\omega),v^{k_j}(t,\omega)\rangle_{\mathbb{R}^n},\varphi\rangle
\nonumber\\
&=&\langle\langle A^iu(t,\omega),\theta(t,\omega)\rangle_{\mathbb{R}^n},\varphi\rangle,
\end{eqnarray}
where the 1$^{\rm st}$ identity above follows from \eqref{eqn-105}$_4$, the 2$^{\rm nd}$  follows from \eqref{eqn-104}
the 3$^{\rm rd}$  follows from  \eqref{eqn-10.29} and Since by Proposition \ref{prop-102}. Summarizing,  we proved that

\begin{equation}\label{eqn-M^i}
M^i(t,\omega)=\langle A^iu(t,\omega),\theta(t,\omega)\rangle_{\mathbb{R}^n},\qquad (t,\omega)\in\mathbb{R}_+\times\Omega,
\end{equation}

Let $Q:=Q_u$   be the event  introduced in Proposition \ref{prop-set-Q_u} and let  $\textbf{u}$ be  the process  introduced in Definition
\ref{def-bfu}. Let us assume that $\omega \in Q$. Then, since $\mathbf u$ is an $M$-valued process and hence  uniformly bounded, it follows from \eqref{eqn-M^i} that   for every $R>0$ we have
\begin{eqnarray}\label{redi}
\sup_{t\in [0,R]}\Vert M^i(t,\omega)\Vert_{L^2(B_R)}&=&\sup_{t\in [0,R]}\Vert\langle\theta(t,\omega),A^i\mathbf u(t,\omega)\rangle_{\mathbb{R}^n}\Vert_{L^2(B_R)}\nonumber
\\
&\leq &c\sup_{t\in [0,R]}\Vert\theta(t,\omega)\Vert_{L^2(B_R)}\nonumber
\\
&\leq &c\sup_{k\ge 0}\Vert v^k(\omega)\Vert_{C([0,R];L^2(B_R))}
<\infty.
\end{eqnarray}
So we conclude that the process $\mathbf{M}_i:=1_{Q_u}M^i$ takes values in the space $L^2_{\textrm{loc}}$. Consequently, in view of Lemma \ref{lem-an-4},  as $M^i\in C_w(\mathbb{R}_+;L^{\newr}_{\textrm{loc}})$,
and  by (\ref{redi}), $\mathbf{M}_i$ has weakly continuous paths in $L^2_{\textrm{loc}}$.
In this way the proof of one part of claim (i) is complete. Later on we will deal with the adaptiveness of the process $1_{Q_u}M^i$.

In the next part of the proof we shall deal with \eqref{eqn-M^i=VA^i}. For this aim  let us  observe that for any $\varphi\in\mathscr D$ and $(t,\omega)\in\mathbb{R}_+\times\Omega$, by \eqref{eqn-105}$_{1,2}$ and Proposition \ref{prop-wis} we infer that
\begin{eqnarray*}
&&\hspace{-2truecm}\lefteqn{\int_0^t\left\langle\varphi,M^i(s)\right\rangle\,ds=\lim_{k\to\infty}\int_0^t\left\langle\varphi,\left\langle v^k(s),A^iu^k(s)\right\rangle_{\mathbb{R}^n}\right\rangle\,ds}
\\
&=&\lim_{k\to\infty}\int_0^t\left\langle\varphi,\left\langle v^k(s),A^iu(s)\right\rangle_{\mathbb{R}^n}\right\rangle\,ds
=\int_0^t\left\langle\varphi,\left\langle v(s),A^iu(s)\right\rangle_{\mathbb{R}^n}\right\rangle\,ds.
\end{eqnarray*}
 Hence we infer that for every $\omega\in\Omega$
\begin{equation}\label{ider2}
\left\langle v(t,\omega),A^iu(t,\omega)\right\rangle_{\mathbb{R}^n}=M^i(t,\omega) \mbox{ for almost every }t\ge 0
\end{equation}
and hence \eqref{eqn-M^i=VA^i} follows.

What concerns the proofs of \eqref{reinf} and \eqref{reing} let us observe that we only need to prove the former one as the proof
of the latter is identical. Moreover,  in view of formulae  \eqref{eqn-710},  \eqref{eqn-f} and \eqref{ider2} we need to deal with with
the following  three limits, weakly  in  $L^2_{\textrm{loc}}$, on $\mathbb{R}_+\times \Omega$.

\begin{equation}\label{conv-00}
\lim_{k\to\infty}\left\langle f^{m_k}_0(u^k)v^k,A^iu^k\right\rangle_{\mathbb{R}^n} =f_0(u)M^i,
\end{equation}

\begin{equation}\label{conv-01}
\lim_{k\to\infty}\left\langle f^{m_k}_l(u^k)\partial_{x_l} u^k,A^iu^k\right\rangle_{\mathbb{R}^n}=\left\langle f_l(u)\partial_{x_l}u,A^iu\right\rangle_{\mathbb{R}^n},
\end{equation}

\begin{equation}\label{conv-0d+1}
\lim_{k\to\infty}\left\langle f^{m_k}_{d+1}(u^k),A^iu^k\right\rangle_{\mathbb{R}^n}=\left\langle f_{d+1}(u),A^iu\right\rangle_{\mathbb{R}^n}.
\end{equation}
The last of these three follows easily from  Proposition \ref{prop-102} (according to which for every $t\in \mathbb{R}_+$ and every $R>0$  $u^k(t)\to u(t)$ in $L^2(B_{R})$) and the convergence \eqref{eqn-conv_f^m_itof_i}.
The proofs of middle ones are more complex but can be done in a similar (but simpler) way to the proof of the first (which we present below).

To prove \eqref{conv-00} let us choose $R>0$ such that \eqref{eqn-supp-aprox} holds, in particular
$\bigcup_{m\in\mathbb{N}}  \supp(f^m_0)  \subset  B(0,R)$.  Since $f_0^m \to f_0$ uniformly $f_0^m(u)M^i\in L^2_{\textrm{loc}}$ by \eqref{redi}, by \eqref{eqn-105}$_4$ and
 $$|f_0^{m_k}(u^k)\langle v^k,A^iu^k\rangle_{\mathbb{R}^n}|\leq c_R\mathbf 1_{B_R}(u^k)|v^k|_{\mathbb{R}^n},$$
by the Lebesgue dominated Theorem  we infer  that for every  $\varphi\in L^2_{\textrm{comp}}$
\begin{equation}\label{eqn-10.033}
\deld{\int_{\mathbb{R}^d}\varphi f_0^{m_k}(u^k) \langle v^k,A^iu^k \rangle_{\mathbb{R}^n}\,dx
=} \lim_{k\to\infty}\int_{\mathbb{R}^d}\varphi f_0^{m_k}(u^k) M^i_k \,dx
= \int_{\mathbb{R}^d}\varphi f_0(u)M^i\,dx \;\mbox{ on } \mathbb{R}_+\times\Omega.
\end{equation}
what proves \eqref{conv-00}.

 As mentioned earlier, this  proves  \eqref{reinf}.

To prove that the process $\mathbf{M}^i$ is adapted let us first notice that, by (\ref{ider2}) and Lemma \ref{thick-V}, for almost every $t\ge 0$,  $M^i(t)=\left\langle\mathbf V(t),A^iu(t)\right\rangle_{\mathbb{R}^n}$ almost surely, hence $M^i(t):\Omega\to L^{\newr}_{\textrm{loc}}$ is $\mathscr{F}_t$-measurable for almost every $t\ge 0$. Also, since $\langle v_0,A^iu(0)\rangle=M^i(0)$ on $\Omega$, the random variable $M^i(0):\Omega\to L^{\newr}_{\textrm{loc}}$ is $\mathscr{F}_0$-measurable.
Now, if $\varphi\in L^r_{\textrm{comp}}$  then $\langle\varphi,M^i\rangle$ is continuous hence $M^i$ is
$\mathbb{F}$-adapted in $L^{\newr}_{\textrm{loc}}$.

Finally we will prove the 2nd and 3rd identities in claim (ii). For this aim let $H$ be the function introduced in the proof of Lemma \ref{thick-V}. Then by (\ref{vsolu}) for every $q\in\mathbb{Q}_+$,
$$H^\prime(\mathbf u(q,\omega))\bold v(q,\omega)=\bold v(q,\omega) \mbox{  almost surely}.
  $$
  Since both the left hand side and  the right hand side of the last equality are weakly continuous in $L^2_{\textrm{loc}}$, the proof of both identities  \eqref{eqn-v_0} and \eqref{eqn-v=V}  is complete. In conclusion, the proof of Lemma \ref{verin} is finished.
\end{proof}

The proof of the following lemma will be given jointly with the proof of Lemma \ref{lem-essee-1}.

\begin{lemma}\label{lem-essee-2} The processes $(w_l)_{l=1}^\infty$ are independent real $\mathbb{F}$-Wiener processes.
\end{lemma}

To formulate the next result let us  define a distribution-valued process $P^i$ by formula
\begin{eqnarray}\label{eqn-P^i}
P^i(t)&=&\mathbf M^i(t)-\mathbf M^i(0)-\sum_{j=1}^d\partial_{x_j}\left[\int_0^t\left\langle\partial_{x_j}\mathbf u(\tau),A^i\mathbf u(\tau)\right\rangle_{\mathbb{R}^n}\,d\tau\right]
\\
&-&\int_0^t\left\langle f(\mathbf u(\tau),\mathbf V(\tau),\nabla\mathbf u(\tau)),A^i\mathbf u(\tau)\right\rangle_{\mathbb{R}^n}\,d\tau,\qquad t\geq 0\nonumber.
\end{eqnarray}
Since the integrals in \eqref{eqn-P^i} are convergent in $L^2_{\textrm{loc}}$, the process $P^i$ takes values in $\mathbb{W}^{-1,2}_R$ for every $R>0$, see Appendix \ref{sec:space_C}.
Moreover, we have the following result.

\begin{lemma}\label{lem-essee-1}
For any  $\varphi\in W^{1,\oldr}_{\textrm{comp}}$ the process $\langle\varphi,P^i\rangle$ is an $\mathbb{F}$-martingale and its  quadratic and cross variations satisfy respectively the following
\begin{eqnarray}
\label{eqn-P^i-qv}
\langle\langle\varphi,P^i\rangle\rangle&=&\int_0^\cdot\Vert[\langle g(\mathbf u(s),\mathbf V(s),\nabla\mathbf u(s)),A^i\mathbf u(s)\rangle_{\mathbb{R}^n}]^*\varphi\Vert^2_{H_\mu}\,ds,
\\
\label{eqn-P^i-cv}
\langle\langle\varphi,P^i\rangle,w_l\rangle &=&\int_0^\cdot\left\langle[\langle g(\mathbf u(s),\mathbf V(s),\nabla\mathbf u(s)),A^i\mathbf, u(s)\rangle_{\mathbb{R}^n}]^*\varphi,e_l\right\rangle_{H_\mu}\,ds
\end{eqnarray}

where $g^*\varphi$ denotes the only element in $H_\mu$ such that $\langle g\xi,\varphi\rangle=\langle\xi,g^*\varphi\rangle_{H_\mu}$, $\forall\xi\in H_\mu$.
\end{lemma}

\begin{proof}Let us take a function    $\varphi\in L^2_{\textrm{comp}}$ and  $R>0$. Then, by employing the  argument used earlier  in the paragraph between \eqref{reok} and \eqref{zaklomez}, the  following three maps $\fourIdx1Rik{N}$, $\fourIdx2{\varphi}ik{N}$ and \fourIdx3{\varphi,l}ik{N} are continuous (and hence Borel measurable) between corresponding  Polish spaces:

$$
\fourIdx1Rik{N}:C(\mathbb{R}_+,\mathscr H_{\textrm{loc}})\to C(\mathbb{R}_+,\mathbb{W}^{-1,\newr}_R)
$$
\begin{eqnarray*}
(u,v) &\mapsto& \Big\{t\mapsto
\langle v(t),A^iu(t))\rangle_{\mathbb{R}^n}-\langle v(0),A^iu(0)\rangle_{\mathbb{R}^n}-\sum_{j=1}^d\partial_{x_j}\left[\int_0^t\left\langle\partial_{x_j}u(s),A^iu(s)\right\rangle_{\mathbb{R}^n}\,ds\right]
\\
&-&\int_0^t\left\langle f^{m_k}(z(s),\nabla u(s)),A^iu(s)\right\rangle_{\mathbb{R}^n}\,ds\Big\}
\end{eqnarray*}
$$
\fourIdx2{\varphi}ik{N}:C(\mathbb{R}_+,\mathscr H_{\textrm{loc}})\to C(\mathbb{R}_+)
$$
$$
(u,v)\mapsto\left\{t\mapsto\int_0^t\Vert\left[\langle g^{m_k}(z(s),\nabla u(s)),A^iu(s)\rangle_{\mathbb{R}^n}\right]^*\varphi\Vert_{H_\mu}^2\,ds\right\}
$$
$$
\fourIdx3{\varphi,l}ik{N}:C(\mathbb{R}_+,\mathscr H_{\textrm{loc}})\to C(\mathbb{R}_+)
$$
$$
(u,v)\mapsto\left\{t\mapsto\int_0^t\left\langle\left[\langle g^{m_k}(z(s),\nabla u(s)),A^iu(s)\rangle_{\mathbb{R}^n}\right]^*\varphi,e_l\right\rangle_{H_\mu}\,ds\right\}.
$$
 Thus, if we set $z^k=(u^k,v^k)$, the random variables
$$
\left(\fourIdx1Rik{N}(z^k),\fourIdx2{\varphi}ik{N}(z^k),\fourIdx3{\varphi,l}ik{N}(z^k),z^k,w^k\right)\;\; \mbox{ and } \;\; \left(\fourIdx1Rik{N}(Z^{m_k}),\fourIdx2{\varphi}ik{N}(Z^{m_k}),\fourIdx3{\varphi,l}ik{N}(Z^{m_k}),Z^{m_k},\beta^{m_k}\right)
$$
have the same laws on
$$
\mathscr B(C(\mathbb{R}_+;\mathbb{W}^{-1,\newr}_R))\otimes\mathscr B(C(\mathbb{R}_+))\otimes\mathscr B(C(\mathbb{R}_+))\otimes\mathscr B(C(\mathbb{R}_+;\mathscr H_{\textrm{loc}}))\otimes\mathscr B(C(\mathbb{R}_+;\mathbb{R}^{\mathbb{N}})).
$$
Notice that  by (\ref{reok})
$$
\fourIdx1Rik{N}(Z^{m_k})=\int_0^\cdot\langle g^{m_k}(Z^{m_k}(s),\nabla U^{m_k}(s)),A^iU^{m_k}(s)\rangle_{\mathbb{R}^n}\,dW^{m_k}
$$
in $\mathbb{W}^{-1,\newr}_R$. Hence  for $\varphi\in W^{1,\oldr}_{\textrm{comp}}$ supported in some $B_R\subset \mathbb{R}^d$, $p\ge 2$ and $\delta>0$,
\begin{eqnarray}
&&\mathbb{E}\,1_{\{\|z^k(0)\|_{\mathscr H_{2R}}\le\delta\}}\left[\sup_{t\in[0, R]}\left|\langle\varphi,\fourIdx1Rik{N}(z^k)(t)\rangle\right|^p+\sup_{t\in[0, R]}\left|\fourIdx2{\varphi}ik{N}(z^k)(t)\right|^\frac p2+\sup_{t\in[0, R]}\left|\fourIdx3{\varphi,l}ik{N}(z^k)(t)\right|^p\right]\nonumber
\\
&=&\mathbb{E}^{m_k}\,1_{\{\|Z^{m_k}(0)\|_{\mathscr H_{2R}}\le\delta\}}\Big[\sup_{t\in[0, R]}\left|\langle\varphi,\fourIdx1Rik{N}(Z^{m_k})(t)\rangle\right|^p+\sup_{t\in[0, R]}\left|\fourIdx2{\varphi}ik{N}(Z^{m_k})(t)\right|^\frac p2
\nonumber
\\
&+&\sup_{t\in[0, R]}\left|\fourIdx3{\varphi,l}ik{N}(Z^{m_k})(t)\right|^p\Big] \leq c_{p,R,\delta}\|\varphi\|^p_{L^2(B_R)}.\label{unif-bound-in_Lp1}
\end{eqnarray}
Indeed, by the Burkholder-Gundy-Davis inequality,  and  Lemmata \ref{lem-hsop} and  \ref{basal} all three terms in \eqref{unif-bound-in_Lp1} can be estimated from above by
\begin{eqnarray*}
&& C_{p,R}\mathbb{E}^{m_k}\,1_{\{\|Z^{m_k}(0)\|_{\mathscr H_{2R}}\le\delta\}}\int_0^R\Vert\left[\langle g^{m_k}(Z^{m_k}(s),\nabla U^{m_k}(s)),A^iU^{m_k}(s)\rangle_{\mathbb{R}^n}\rangle\right]^*\varphi\Vert^p_{H_\mu}\,ds
\\
&\leq & c^1_{p,R}\|\varphi\|^p_{L^2(B_R)}\mathbb{E}^{m_k}\,1_{\{\|Z^{m_k}(0)\|_{\mathscr H_{2R}}\le\delta\}}\int_0^R\Vert\langle g^{m_k}(Z^{m_k}(s),\nabla U^{m_k}(s)),A^iU^{m_k}(s)\rangle_{\mathbb{R}^n}\rangle\Vert^p_{L^2(B_R)}\,ds
\\
&\leq & c^2_{p,R}\|\varphi\|^p_{L^2(B_R)}\mathbb{E}^{m_k}\,1_{\{\|Z^{m_k}(0)\|_{\mathscr H_{2R}}\le\delta\}}\int_0^R\left(1+\mathbf e^\frac p2_{0,2R,m_k}(s,Z^{m_k}(s))\right)\,ds\leq c_{p,R,\delta}\|\varphi\|^p_{L^2(B_R)}.
\end{eqnarray*}
 Also
\begin{equation}\label{unif-bound-in_Lp2}
\mathbb{E}^{m_k}\,|\beta^{m_k}_l(r)|^p=\mathbb{E}\,|w^k_l(r)|^p=c_{p,r},\qquad p>0,\quad r\ge 0.
\end{equation}
Let us consider times $s$ and $t$ such that $s< t$. We can always assume that $t<R$. Let $\varphi_1,\dots,\varphi_K$ be  functions belonging to  $L^2_{\textrm{comp}}$ and  let $h:\mathbb{R}^K\times\mathbb{R}^{K\times K}\times [C(\mathbb{R}_+;\mathbb{R}^{\mathbb{N}})]^K\to[0,1]$ be a continuous function.
Let us choose numbers $s_1,\cdots,s_K$ such that  $0\leq s_1\le\dots\leq s_K\leq s$. Let us denote
\begin{eqnarray*}
\tilde a_k&=&h\left(\langle\varphi_{i_1},V^{m_k}(0)\rangle_{i_1\leq K},\langle\varphi_{i_2},U^{m_k}(s_{i_3})\rangle_{i_2,i_3\leq K},(\beta^{m_k}(s_{i_4}))_{i_4\leq K}\right)
\\
a_k&=&h\left(\langle\varphi_{i_1},v^k(0)\rangle_{i_1\leq K},\langle\varphi_{i_2},u^k(s_{i_3})\rangle_{i_2,i_3\leq K},(w^k(s_{i_4}))_{i_4\leq K}\right)
\\
a&=&h\left(\langle\varphi_{i_1},v_0\rangle_{i_1\leq K},\langle\varphi_{i_2},\mathbf u(s_{i_3})\rangle_{i_2,i_3\leq K},(w(s_{i_4}))_{i_4\leq K}\right)
\\
\tilde q_k&=&\tilde a_k\mathbf 1_{\{\|(Z^{m_k}(0))\|_{\mathscr H_{2R}}\le\delta\}},\quad q_k=a_k\mathbf 1_{\{\|(z^k(0))\|_{\mathscr H_{2R}}\le\delta\}},\quad q=a\mathbf 1_{\{\|(\mathbf u(0),\mathbf v_0)\|_{\mathscr H_{2R}}
\leq \delta\}}.
\end{eqnarray*}
Let $\varphi\in W^{1,\oldr}$ has support in $B_R$. Then
\begin{eqnarray}\label{mar1}
&&\hspace{-3truecm}\lefteqn{
\mathbb{E}^{m_k}\,\tilde q_k\left[\langle\varphi,\fourIdx1Rik{N}(Z^{m_k})(t)\rangle-\langle\varphi,\fourIdx1Rik{N}(Z^{m_k})(s)\rangle\right]}
\\&=&\mathbb{E}\,q_k\left[\langle\varphi,\fourIdx1Rik{N}(z^k)(t)\rangle-\langle\varphi,\fourIdx1Rik{N}(z^k)(s)\rangle\right]=0,
\nonumber
\end{eqnarray}
\begin{eqnarray}
&&\hspace{-2truecm}\lefteqn{\mathbb{E}^{m_k}\,\tilde q_k\left[\langle\varphi,\fourIdx1Rik{N}(Z^{m_k})(t)
\rangle^2-\fourIdx2{\varphi}ik{N}(Z^{m_k})(t)-\langle\varphi,\fourIdx1Rik{N}(Z^{m_k})(s)\rangle^2+\fourIdx2{\varphi}ik{N}(Z^{m_k})(s)\right]
}
\\
&=&\mathbb{E}\,q_k\left[\langle\varphi,\fourIdx1Rik{N}(z^k)(t)\rangle^2-\fourIdx2{\varphi}ik{N}(z^k)(t)-\langle\varphi,\fourIdx1Rik{N}(z^k)(s)\rangle^2+\fourIdx2{\varphi}ik{N}(z^k)(s)\right]=0,\label{mar2}
\nonumber
\end{eqnarray}
\begin{eqnarray*}
\mathbb{E}^{m_k}\,\tilde a_k(\beta^{m_k}_l(t)-\beta^{m_k}_l(s))=\mathbb{E}\,a_k(w^k_l(t)-w^k_l(s))=0
\end{eqnarray*}

\begin{eqnarray}
&&\hspace{-3truecm}\lefteqn{\mathbb{E}^{m_k}\,\tilde a_k\left[\beta^{m_k}_l(t)\beta^{m_k}_j(t)-t\delta_{lj}-\beta^{m_k}_l(s)\beta^{m_k}_j(s)+s\delta_{lj}\right]}\nonumber
\\
&=&\mathbb{E}\,a_k\left[w^k_l(t)w^k_j(t)-t\delta_{lj}-w^k_l(s)w^k_j(s)+s\delta_{lj}\right]=0,\nonumber
\end{eqnarray}

\begin{eqnarray}
&&\hspace{-3truecm}\lefteqn{\mathbb{E}^{m_k}\,\tilde q_k\left[\langle\varphi,\fourIdx1Rik{N}(Z^{m_k})(t)\rangle\beta^{m_k}_l(t)-\fourIdx3{\varphi,l}ik{N}(Z^{m_k})(t)\right]}
\nonumber
\\
&-&\mathbb{E}^{m_k}\,\tilde q_k\left[\langle\varphi,\fourIdx1Rik{N}(Z^{m_k})(s)\rangle\beta^{m_k}_l(s)-\fourIdx3{\varphi,l}ik{N}(Z^{m_k})(s)\right]\nonumber
\\
&&\hspace{-2.7truecm}\lefteqn{=\; \mathbb{E}\,q_k\left[\langle\varphi,\fourIdx1Rik{N}(z^k)(t)\rangle w^k_l(t)-\fourIdx3{\varphi,l}ik{N}(z^k)(t)\right]}
\nonumber
\\
&-&\mathbb{E}\,q_k\left[\langle\varphi,\fourIdx1Rik{N}(z^k)(s)\rangle w^k_l(s)-\fourIdx3{\varphi,l}ik{N}(z^k)(s)\right]=0.\label{mar4}
\end{eqnarray}

Next, since by Lemma \ref{lem-hsop}
$$
\|(g^{m_k})^*\varphi\|^2_{H_\mu}=\sum_l\langle g^{m_k},\varphi e_l\rangle^2,\qquad\sum_l\|\varphi e_l\|^2_{L^2(B_R)}\le c\|\varphi\|^2_{L^2(B_R)},
$$
 by applying the compactness of the  embedding of $H^1(B_R)$ in $L^{2r/(r-2)}(B_R)$, Lemma \ref{verin},  property (R2) from the beginning of this section and  the Lebesgue Dominated Convergence Theorem   we infer that the following three limits exist in $C(\mathbb{R}_+)$ almost surely,
\begin{eqnarray*}
\lim_{k\to\infty}\langle\varphi,\fourIdx1Rik{N}(z^k)\rangle&=&\left\langle\varphi,P^i\right\rangle
\\
\lim_{k\to\infty}N^{2i,k}_{\varphi}(z^k)&=&\int_0^\cdot\Vert\left[\left\langle g(\mathbf u(s),\mathbf V(s),\nabla \mathbf u(s)),A^i\mathbf u(s)\right\rangle_{\mathbb{R}^n}\right]^*\varphi\Vert_{H_\mu}^2\,ds=:P^{2i}_\varphi,
\\
\lim_{k\to\infty}\fourIdx3{\varphi,l}ik{N}(z^k)&=&\int_0^\cdot\left\langle\left[\left\langle g(\mathbf u(s),\mathbf V(s),\nabla \mathbf u(s)),A^i\mathbf u(s)\right\rangle_{\mathbb{R}^n}\right]^*\varphi,e_l\right\rangle_{H_\mu}\,ds=:P^{3i}_{\varphi,l}.
\end{eqnarray*}

 So, in view of the uniform boundedness (\ref{unif-bound-in_Lp1}), (\ref{unif-bound-in_Lp2}) in every $L^p(\Omega)$, the integrals (\ref{mar1})-(\ref{mar4}) converge as $k\toup \infty$ provided that $\mathbb{P}\,\{\|(\mathbf u(0),\mathbf v_0)\|_{\mathscr H_{2R}}
=\delta\}=0$ and we obtain
\begin{eqnarray*}
\mathbb{E}\,q\langle\varphi,P^i(t)\rangle&=&\mathbb{E}\,q\langle\varphi,P^i(s)\rangle
\\
\mathbb{E}\,q\left[\langle\varphi,P^i(t)\rangle^2-P^{2i}_\varphi(t)\right]&=&\mathbb{E}\,q\left[\langle\varphi,P^i(s)\rangle^2-P^{2i}_\varphi(s)\right]
\\
\mathbb{E}\,aw_l(t)&=&\mathbb{E}\,aw_l(s)
\\
\mathbb{E}\,a\left[w_l(t)w_j(t)-t\delta_{lj}\right]&=&\mathbb{E}\,a\left[w_l(s)w_j(s)-s\delta_{lj}\right]
\\
\mathbb{E}\,q\left[\langle\varphi,P^i(t)\rangle w_l(t)-P^{3i}_{\varphi,l}(t)\right]&=&\mathbb{E}\,q\left[\langle\varphi,P^i(s)\rangle w_l(s)-P^{3i}_{\varphi,l}(s)\right].
\end{eqnarray*}
Hence, in view of Corollary \ref{aml1},
\begin{eqnarray*}
\mathbb{E}\,\left[a_\delta\langle\varphi,P^i(t)\rangle|\mathscr{F}_s\right]&=&a_\delta\langle\varphi,P^i(s)\rangle
\\
\mathbb{E}\,\left\{a_\delta\left[\langle\varphi,P^i(t)\rangle^2-P^{2i}_\varphi(t)\right]|\mathscr{F}_s\right\}&=&a_\delta\left[\langle\varphi,P^i(s)\rangle^2-P^{2i}_\varphi(s)\right]
\\
\mathbb{E}\,[w_l(t)|\mathscr{F}_s]&=&w_l(s)
\\
\mathbb{E}\,\left[w_l(t)w_j(t)-t\delta_{lj}|\mathscr{F}_s\right]&=&w_l(s)w_j(s)-s\delta_{lj}
\\
\mathbb{E}\,\left\{a_\delta\left[\langle\varphi,P^i(t)\rangle w_l(t)-P^{3i}_{\varphi,l}(t)\right]|\mathscr{F}_s\right\}&=&a_\delta\left[\langle\varphi,P^i(s)\rangle w_l(s)-P^{3i}_{\varphi,l}(s)\right]
\end{eqnarray*}
where $a_\delta=1_{[\|(\mathbf u(0),\mathbf v_0)\|_{\mathscr H_{2R}}\le\delta]}$.

 Therefore we proved that $w_1,w_2,\dots$ are independent $\mathbb{F}$-wiener processes, $a_\delta\langle\varphi,P^i\rangle$ is an
$\mathbb{F}$-martingale on $[0,T]$, and the quadratic and the cross variations satisfy
 $$\langle a_\delta\langle\varphi,P^i\rangle\rangle=a_\delta P^{2i}_\varphi \mbox{ and } \langle a_\delta\langle\varphi,P^i\rangle,w_l\rangle=a_\delta P^{3i}_{\varphi,l} \mbox{ on }[0,R].$$
 In order to finish the proof  we introduce the following $\mathbb{F}$-stopping times
$$
\tau_l=\inf\,\left\{t\in[0, R]:\sup_{s\in[0, t]}|\langle\varphi,P^i\rangle(t)|+\int_0^t\Vert\left[\langle g(\mathbf u(s),\mathbf V(s),\nabla\mathbf u(s)),A^i\mathbf u(s)\rangle_{\mathbb{R}^n}\right]^*\varphi\Vert^2_{H_\mu}\,ds\ge l\right\}.
$$
 By letting $\delta\toup \infty$ we deduce that $(\tau_l)$ localizes $\langle\varphi,P^i\rangle$, $P^{2i}_\varphi$ and $P^{3i}_{\varphi,l}$ on $[0,R]$.  The result now follows by letting $R\toup \infty$.
\end{proof}

\begin{proposition}\label{prop-Wiener} Let $(e_l)_{l=1}^\infty$ be an ONB of the RKHS $H_\mu$ and let us  set
\begin{equation}\label{eqn-1058} W\psi=\sum_{l=1}^\infty w_le_l(\psi),\;\; \psi\in\mathscr S(\mathbb{R}^d).
 \end{equation}
 Then $W$ is a spatially homogeneous $\mathbb{F}$-Wiener process with spectral measure $\mu$, and for every  function $\varphi\in H^1_{\textrm{comp}}$ the following equality holds almost surely
\begin{eqnarray}
\label{eqn-1059}
\left\langle\varphi,\mathbf M^i(t)\right\rangle&=&\left\langle\varphi,\mathbf M^i(0)\right\rangle-\sum_{k=1}^d\left\langle\partial_{x_k}\varphi,\int_0^t\left\langle\partial_{x_k}\mathbf u (s),A^i\mathbf u(s)\right\rangle_{\mathbb{R}^n}\,ds\right\rangle
\\
\nonumber
&+&\left\langle\varphi,\int_0^t\left\langle f(\mathbf u(s),\mathbf V(s),\nabla\mathbf u(s)),A^i\mathbf u(s)\right\rangle\,d s\right\rangle
\\
\nonumber
&+&\left\langle\varphi,\int_0^t\left\langle g(\mathbf u(s),\mathbf V(s),\nabla\mathbf u(s)),A^i\mathbf u(s)\right\rangle\,dW(s)\right\rangle,\qquad t\ge 0.
\end{eqnarray}

\begin{proof} By Lemma \ref{lem-essee-2} we infer that $W\varphi$ is an $\mathbb{F}$-Wiener process and
$$
\mathbb{E}\,|W_t\varphi|^2=t\sum_l|e_l(\varphi)|^2=t\|\widehat\varphi\|_{L^2(\mu)}^2.
$$
Hence  $W$ is a spatially homogeneous $\mathbb{F}$-Wiener process with spectral measure $\mu$.

Let now $P^i$ be the process defined by formula \eqref{eqn-P^i}. Let  $\varphi\in W^{1,\oldr}_{\textrm{comp}}$.
In order to prove equality \eqref{eqn-1060} it is enough to show that
\begin{equation}\label{eqn-1060}
\left\langle\langle\varphi,P^i\rangle-\int\langle\langle g(\mathbf u(s),\mathbf V(s),\nabla\mathbf u(s)),A^i\mathbf u(s)\rangle_{\mathbb{R}^n}\,dW,\varphi\rangle\right\rangle=0.
\end{equation}

Then we have

\begin{eqnarray*}
&&\left\langle\langle\varphi,P^i\rangle,\int\langle\langle g(\mathbf u(s),\mathbf V(s),\nabla\mathbf u(s)),A^i\mathbf u(s)\rangle_{\mathbb{R}^n}\,dW,\varphi\rangle\right\rangle
\\
&=&\sum_l\left\langle\langle\varphi,P^i\rangle,\int\langle\langle g(\mathbf u(s),\mathbf V(s),\nabla\mathbf u(s)),A^i\mathbf u(s)\rangle_{\mathbb{R}^n}e_l,\varphi\rangle\,dw_l\right\rangle
\\
&=&\sum_l\int\langle\langle g(\mathbf u(s),\mathbf V(s),\nabla\mathbf u(s)),A^i\mathbf u(s)\rangle_{\mathbb{R}^n}e_l,\varphi\rangle\,d\left\langle\langle\varphi,P^i\rangle,w_l\right\rangle
\\
&=&\sum_l\int\langle\langle g(\mathbf u(s),\mathbf V(s),\nabla\mathbf u(s)),A^i\mathbf u(s)\rangle_{\mathbb{R}^n}e_l,\varphi\rangle^2\,ds
\\
&=&\int\Vert[\langle g(\mathbf u(s)\mathbf V(s),\nabla\mathbf u(s)),A^i\mathbf u(s)\rangle_{\mathbb{R}^n}]^*\varphi\Vert^2_{H_\mu}\,ds
\\
&=&\left\langle\langle\varphi,P^i\rangle\right\rangle
\\
&=&\left\langle\int\langle\langle g(\mathbf u(s),\mathbf V(s),\nabla\mathbf u(s)),A^i\mathbf u(s)\rangle_{\mathbb{R}^n}\,dW,\varphi\rangle\right\rangle
\end{eqnarray*}
so \eqref{eqn-1060} follows. The proof is complete.

\end{proof}

\begin{lemma}\label{lem-identification} The $L^2_{\textrm{loc}}$-valued process $\bold v$ introduced in (\ref{vsolu}) is $\mathbb{F}$-adapted and weakly continuous. Moreover,  $\bold v(t)\in T_{\mathbf u(t)}M$ for every $t\ge 0$ almost surely and for every  $\varphi\in\mathscr D(\mathbb{R}^d)$
\begin{eqnarray}\label{eqn-1061}
\langle\bold v(t),\varphi\rangle&=&\langle\bold v(0),\varphi\rangle+\int_0^t\left\langle\mathbf u(s),\Delta\varphi\right\rangle\,ds+\int_0^t\left\langle\bold S_{\mathbf u(s)}\left(\bold v(s),\bold v(s)\right),\varphi\right\rangle
\\
&-&\sum_{k=1}^d\int_0^t\left\langle\bold S_{\mathbf u(s)}\left(\partial_{x_k}\mathbf u(s),\partial_{x_k}\mathbf u(s)\right),\varphi\right\rangle+\int_0^t\left\langle f(\mathbf u(s),\mathbf v(s),\nabla\mathbf u(s)),\varphi\right\rangle\,ds
\nonumber
\\
&+&\int_0^t\left\langle g(\mathbf u(s),\mathbf v(s),\nabla\mathbf u(s))\,dW,\varphi\right\rangle
\nonumber
\end{eqnarray}
almost surely for every $t\ge 0$.
\end{lemma}

\begin{proof} Obviously  the process $\bold v$ is $L^2_{\textrm{loc}}$-valued. The $\mathbb{F}$-adaptiveness  and the weak continuity of $\bold v$  follows from the definition (\ref{vsolu}) and Lemma \ref{verin}.

In order to prove the equality \eqref{eqn-1061} let us take a test function $\varphi\in\mathscr D(\mathbb{R}^d)$ and functions $h_{ij}$  as in the Assumption \textbf{M4}. Then we consider vector fields $Y^i$, $i=1,\cdots,n$ defined by formula \eqref{eqn-Y^i}, i.e. $Y^i(x)=\sum_{j=1}^Nh_{ij}(x)A^jx$. Let $\mathbf{M}^i$ be the process  introduced in Lemma \ref{verin} and satisfies the identity \eqref{eqn-1061}. Then  by applying Lemma \ref{lem-ito} to  the processes $\mathbf u$ and  $\mathbf{M}^i$ and the vector field $Y^i$ we get the following equality
\begin{eqnarray}
\sum_{i=1}^N\Big\langle\mathbf M^i(t)Y^i(\mathbf u(t)),\varphi\Big\rangle&=&\sum_{i=1}^N\Big\langle\mathbf M^i(0)Y^i(\mathbf u(0)),\varphi\Big\rangle+\sum_{i=1}^N\int_0^t\Big\langle\mathbf  M^i(s)(d_{\mathbf u(s)}Y^i)\big(\mathbf V(s)\big),\varphi\Big\rangle\,ds\nonumber
\\
&-&\sum_{i=1}^N\sum_{l=1}^d\int_0^t\Big\langle\big\langle\partial_{x_l}\mathbf u(s),A^i\mathbf u(s)\big\rangle_{\mathbb{R}^n}Y^i(\mathbf u(s)),\partial_{x_l}\varphi\Big\rangle\,ds
\label{eqn-M^iY^i}
\\
&-&\sum_{i=1}^N\sum_{l=1}^d\int_0^t\Big\langle\big\langle\partial_{x_l}\mathbf u(s),A^i\mathbf u(s)\big\rangle_{\mathbb{R}^n}(d_{\mathbf u(s)}Y^i)\big(\partial_{x_l}\mathbf u(s)\big),\varphi\Big\rangle\,ds
\nonumber\\
&+&\sum_{i=1}^N\int_0^t\Big\langle\big\langle f(\mathbf u(s),\mathbf V(s),\nabla\mathbf u(s)),A^i\mathbf u(s)\big\rangle_{\mathbb{R}^n}Y^i(\mathbf u(s)),\varphi\Big\rangle\,ds\nonumber
\\
&+&\sum_{i=1}^N\int_0^t\Big\langle\big\langle g(\mathbf u(s),\mathbf V(s),\nabla\mathbf u(s)),A^i\mathbf u(s)\big\rangle_{\mathbb{R}^n}Y^i(\mathbf u(s))\,dW,\varphi\Big\rangle\nonumber
\end{eqnarray}
$\mathbb{P}$-almost surely for every $t\ge 0$.

Next, by identity \eqref{vsolu} in Lemma \ref{verin}, we have, for each $t\geq 0$ and $\omega \in \Omega$,
\begin{equation}
\label{eqn-1063}
\sum_{i=1}^N\left\langle\mathbf M^i(t)Y^i(\mathbf u(t)),\varphi\right\rangle=\left\langle\mathbf v(t,\omega),\varphi\right\rangle.
\end{equation}
and by identity \eqref{ass-M4_2} we have, for each $s\geq 0$ and $\omega \in \Omega$,
\begin{equation}
\label{eqn-1064}
\sum_{i=1}^N\sum_{l=1}^d\Big\langle\big\langle\partial_{x_l}\mathbf u(s),A^i\mathbf u(s)\big\rangle_{\mathbb{R}^n}Y^i(\mathbf u(s)),\partial_{x_l}\varphi\Big\rangle=\sum_{l=1}^d \Big\langle\partial_{x_l}\mathbf u(s),\partial_{x_l}\varphi\Big\rangle.
\end{equation}
Furthermore, from Lemma \ref{lem-2ff} we infer that  for each $s\geq 0$ and $\omega \in \Omega$,
\begin{equation}\label{eqn-1065}
\sum_{i=1}^N\sum_{l=1}^d \Big\langle\big\langle\partial_{x_l}\mathbf u(s),A^i\mathbf u(s)\big\rangle_{\mathbb{R}^n}
(d_{\mathbf u(s)}Y^i)\big(\partial_{x_l}\mathbf u(s)\big),\varphi\Big\rangle
=\sum_{l=1}^d \Big\langle \bold S_{\mathbf u(s)}\Big(\partial_{x_l}\mathbf u(s),\partial_{x_l}\mathbf u(s)\Big),\varphi\Big\rangle.
\end{equation}
Similarly, by the identity \eqref{eqn-M^i=VA^i}, identity \eqref{eqn-2ff} in  Lemma \ref{lem-2ff} and Lemma \ref{thick-V}, we infer that for a.e. $s\geq 0$,
\begin{eqnarray}\label{eqn-1066} \sum_{i=1}^N \Big\langle\mathbf  M^i(s)(d_{\mathbf u(s)}Y^i)\big(\mathbf V(s)\big),\varphi\Big\rangle
 &=&   \sum_{i=1}^N \Big\langle \big\langle \mathbf V(s) ,A^i\mathbf u(s)\big\rangle_{\mathbb{R}^n}(d_{\mathbf u(s)}Y^i)\big(\mathbf V(s)\big)
 ,\varphi\Big\rangle
\nonumber
\\
=\Big\langle\bold S_{\mathbf u(s)}\big(\mathbf V(s),\mathbf V(s)\big),\varphi\Big\rangle
&=&\Big\langle\bold S_{\mathbf u(s)}\big(\mathbf v(s),\mathbf v(s)\big),\varphi\Big\rangle
\end{eqnarray}
holds a.s. Moreover, by a similar argument based on  \eqref{ass-M4_2} we can deal with  the integrands of the last two terms on the RHS of \eqref{eqn-M^iY^i}. Indeed by \eqref{eqn-Y^i}  we  get
\begin{eqnarray}
\sum_{i=1}^N\Big\langle\big\langle f(\mathbf u(s),\mathbf V(s),\nabla\mathbf u(s)),A^i\mathbf u(s)\big\rangle_{\mathbb{R}^n}Y^i(\mathbf u(s)),\varphi\Big\rangle= \Big\langle f(\mathbf u(s),\mathbf V(s),\nabla\mathbf u(s)),\varphi\Big\rangle,
\\
\sum_{i=1}^N \Big\langle\big\langle g(\mathbf u(s),\mathbf V(s),\nabla\mathbf u(s)),A^i\mathbf u(s)\big\rangle_{\mathbb{R}^n}Y^i(\mathbf u(s)),\varphi\Big\rangle= \Big\langle  g(\mathbf u(s),\mathbf V(s),\nabla\mathbf u(s)),\varphi\Big\rangle.
\end{eqnarray}
Summing up, we infer from the equality \eqref{eqn-M^iY^i} and the other equalities which follow it that for every $t\geq 0$ almost surely

\begin{eqnarray}
\left\langle\mathbf v(t),\varphi\right\rangle &=&  \left\langle\mathbf v(0),\varphi\right\rangle
-\int_0^t \sum_{l=1}^d \Big\langle\partial_{x_l}\mathbf u(s),\partial_{x_l}\varphi\Big\rangle\, ds\\
 &-& \int_0^t \sum_{l=1}^d \Big\langle \bold S_{\mathbf u(s)}\Big(\partial_{x_l}\mathbf u(s),\partial_{x_l}\mathbf u(s)\Big),\varphi\Big\rangle \, ds
\nonumber\\
&+&\int_0^t \Big\langle f(\mathbf u(s),\mathbf v(s),\nabla\mathbf u(s)),\varphi\Big\rangle+ \int_0^t \Big\langle  g(\mathbf u(s),\mathbf v(s),\nabla\mathbf u(s))\,dW(s),\varphi\Big\rangle.
\nonumber\end{eqnarray}

This concludes the proof of Lemma \ref{lem-identification}.
\end{proof}

To conclude the proof of the existence of a solution, i.e. the proof of Theorem \ref{thm-main} let us observe that the above equality
 is nothing else but \eqref{sol_weak_2}. Moreover, \eqref{sol_weak_1} follows from \eqref{feq} and \eqref{vsolu}. This proves that if the process $\mathbf z:=\big(\mathbf u, \mathbf v\big)$  then $(\Omega,\mathscr{F},\mathbb{F},\mathbb{P},W,z)$ a weak solution to equation  \eqref{equa1}.

\end{proposition}

\appendix


\section{The Jakubowski's version of the Skorokhod representation theorem}\label{sec:Jakubowski}

\begin{theorem}\label{thm-skoro} Let $X$ be a topological space such that there exists a sequence $\{f_m\}$ of continuous functions $f_m:X\to\mathbb{R}$ that separate points of $X$.  Let us  denote by $\mathscr S$ the $\sigma$-algebra generated by the maps $\{f_m\}$. Then
\begin{trivlist}
\item[(j1)] every compact subset of $X$ is metrizable,
\item[(j2)] every Borel subset of a $\sigma$-compact set in $X$ belongs to $\mathscr S$,
\item[(j3)] every probability measure supported by a $\sigma$-compact set in $X$ has a unique Radon extension to the Borel $\sigma$-algebra on $X$,
\item[(j4)] if $(\mu_m)$ is a tight sequence of probability measures on $(X,\mathscr S)$, then there exists a subsequence $(m_k)$, a probability space $(\Omega,\mathscr{F},\mathbb{P})$ with $X$-valued Borel measurable random variables $X_k$, $X$ such that $\mu_{m_k}$ is the law of $X_k$ and $X_k$ converge almost surely to $X$. Moreover, the law of $X$ is a Radon measure.
\end{trivlist}
\end{theorem}

\begin{proof} See \cite{Jakub_1997}.
\end{proof}

\begin{corollary}\label{fiza} Under the assumptions of Theorem \ref{thm-skoro}, ff $Z$ is a Polish space and $b:Z\to X$ is a continuous injection, then $b[B]$ is a Borel set whenever $B$ is Borel in $Z$.
\end{corollary}

\begin{proof} See Corollary A.2 in \cite{o5}.
\end{proof}

\section{The space $L^\infty_{\textrm{loc}}(\mathbb{R}_+;L^2_{\textrm{loc}})$}\label{sec:space_L}

 Let $\mathbb{L}=L^\infty_{\textrm{loc}}(\mathbb{R}_+;L^2_{\textrm{loc}})$  be the space of equivalence classes $[f]$ of all measurable functions  $f:\mathbb{R}_+\to L^2_{\textrm{loc}}= L^2_{\textrm{loc}}(\mathbb{R}^d;\mathbb{R}^n)$ such that $\|f\|_{L^2(B_n)}\in L^\infty(0,n)$ for every $n\in\mathbb{N}$. The space $\mathbb{L}$ is equipped with  the locally convex topology generated by functionals
\begin{equation}\label{protgen}
f\mapsto\int_0^n\int_{B_n}\langle g(t,x),f(t,x)\rangle_{\mathbb{R}^n}\,dx\,dt,
\end{equation}
where $n\in\mathbb{N}$ and $g\in L^1(\mathbb{R}_+,L^2(\mathbb{R}^d))$.

Let us  also define a space
\begin{equation}\label{eqn-space-Y_m}
Y_m=L^1((0,m),L^2(B_m)),
 \end{equation}
Let us recall that $L^\infty((0,m),L^2(B_m))=Y_m^*$. Consider the following natural restriction maps

\begin{eqnarray}\label{eqn-B.03}
\pi_m&:&L^2(\mathbb{R}^d) \ni g\mapsto  g|_{B_m} \in L^2(B_m),\\
\label{eqn-B.04}
l_m &:&\mathbb{L} \ni f\mapsto (\pi_m\circ f)|_{[0,m]} \in (Y_m^*,w^*).
\end{eqnarray}
 The following results describe some properties of the space $\mathbb{L}$.

\begin{lemma}\label{lemi} A map $l=(l_m(f)\big)_{m\in\mathbb{N}}:\mathbb{L}\to \prod_{m\in\mathbb{N}}(Y^*_m,w^*)$
is a homeomorphism onto a closed subset of $\prod_{m\in\mathbb{N}}(Y^*_m,w^*)$.
\end{lemma}

\begin{proof} The proof is straightforward.
\end{proof}

\begin{corollary}\label{compero} Given any sequence $(a_m)$ of positive numbers, the set
\begin{equation}\label{eqn-B.05}
\{f\in\mathbb{L}:\|f\|_{L^\infty((0,m),L^2(B_m))}\leq a_m,\,m\in\mathbb{N}\}
\end{equation}
is compact in $\mathbb{L}$.
\end{corollary}

\begin{proof} The proof follows immediately from Lemma \ref{lemi} and the Banach-Alaoglu theorem since a product of compacts is a compact by the Tychonov theorem.
\end{proof}

\begin{corollary}\label{coli} The Skorokhod representation theorem \ref{thm-skoro} holds for every tight sequence of probability measures defined on $\big(\mathbb{L},\sigma(\mathbb{L}^*)\big)$,  where the $\sigma$-algebra $\sigma(\mathbb{L}^*)$ is the $\sigma$-algebra on $\mathbb{L}$ generated by $\mathbb{L}^*$.
\end{corollary}

\begin{proof}
 Since each $Y_m$ is a separable Banach space,  there exists a sequence $(j_{m,k})_{k}$,  such that  each $j_{m,k}:(Y^*_m,w^*)\to\mathbb{R}$ is a  continuous function and  $(j_{m,k})_{k}$  separate points of $Y^*_m$. Consequently, such a separating sequence of continuous functions exists for product space $\prod(Y^*_m,w^*)$, and, by Lemma \ref{lemi}, for the $\mathbb{L}$ as well. Existence of a separating sequence of continuous functions is sufficient for the Skorokhod representation theorem to hold by the Jakubowski theorem \cite{Jakub_1997}.
\end{proof}

\begin{proposition}\label{prop-wis} Let $\bar \xi$ be an $\mathbb{L}$-valued random variable. Then there exists a measurable $L^2_{\textrm{loc}}$-valued process $\xi$  such that for every $\omega\in\Omega$,
\begin{equation}\label{eqn-B.06}[\xi(\cdot,\omega)]=\bar\xi(\omega).
\end{equation}
\end{proposition}

\begin{proof} Let $\big(\varphi_n\big)_{n=1}^\infty$ be an approximation of identity on $\mathbb{R}$. Let us fix $t\geq 0$ and $n\in\mathbb{N}^\ast$. Then the linear operator
\begin{equation}\label{eqn-B.07}
I_n(t):\mathbb{L}\ni f\mapsto\int_0^\infty\varphi_n(t-s)f(s)\,ds \in L^2_{\textrm{loc}}(\mathbb{R}^d)
\end{equation}
 is well defined and for all $\psi\in (L^2_{\textrm{loc}}(\mathbb{R}^d))^*=L^2_{\textrm{comp}}(\mathbb{R}^d)$ and $t\geq 0$,
the function $\psi\circ I_n(t):\mathbb{L}\to\mathbb{R}$ is continuous. Hence in view of
  Corollary \ref{aml1}  the map $I_n(t)$ is  Borel measurable.
 We put
\begin{equation}\label{eqn-B.09}
I: \mathbb{L}\ni f\mapsto \begin{cases}\lim_{n\to\infty}I_n(t)(t),&  \mbox{provided the limit in }  L^2_{\textrm{loc}}(\mathbb{R}^d) \mbox{ exists},\\
 0, & \mbox{ otherwise}.
 \end{cases}
  \end{equation}
Then (by   employing  the Lusin Theorem \cite{Rudin_1987_RCA} in case (ii)) we infer that  given $f\in\mathbb{L}$
 \begin{trivlist}
 \item[(i)]  the map $\mathbb{R}_+\ni t\mapsto I_n(t)f\in L^2_{\textrm{loc}}$ is continuous, and
 \item[(ii)]  $\lim_{n\to\infty} I_n(t)f$ exists  in $L^2_{\textrm{loc}}$  for almost every $t\in\mathbb{R}_+$ and    $[ I(\cdot){f}]=f$.
 \end{trivlist}
 If we next  define $L^2_{\textrm{loc}}$-valued stochastic processes $\xi_n$, for $n\in\mathbb{N}^\ast$, and $\xi$ by
$\xi_n(t,\omega) = I_n(t)\big(\bar\xi(\omega)\big)$ and  $\xi(t,\omega)=I(t)\big(\bar\xi(\omega)\big)$ for $(t,\omega)\in \mathbb{R}_+\times \Omega$, then
 by (i) above we infer that $\xi_n$  is  continuous  and so   measurable. Hence    the process $\xi$ is  also measurable and   by (ii) above,    given $\omega\in\Omega$,
the function $\{ \mathbb{R}_+\ni t\mapsto \xi(t,\omega)\}$ is a representative of $\bar\xi(\omega)$. The proof is complete.
\end{proof}

\section{The space $C_w(\mathbb{R}_+;W^{k,p}_{\textrm{loc}})$, $k\ge 0$, $1<p<\infty$}\label{sec:space_C}
Let us introduce the spaces, for $l\geq 0$ and with $\frac{1}{p^\prime}+\frac{1}{p}=1$,
\begin{eqnarray*}
W^{l,p}_m & = &  \{f\in W^{l,p}(\mathbb{R}^d;\mathbb{R}^n):f=0\textrm{ on }\mathbb{R}^d\setminus B_m\},
\\
\mathbb{W}^{l,p}_m & = & W^{l,p}(B_m;\mathbb{R}^n), \qquad \qquad \mathbb{W}^{-l,p}_m  = (W^{l,p^\prime}_m)^*.
\end{eqnarray*}
Let us recall that by $(W^{k,p}(B_m),\textrm{w})$ we mean the space $W^{k,p}(B_m)$ endowed with the weak topology.

We now formulate the first of the two main results in this Appendix.

\begin{corollary}\label{cor-dery} Let $a=(a_m)$ be a sequence of positive numbers and let $\gamma>0$, $1<r,p<\infty$, $-\infty<l\leq k$ satisfy
\begin{equation}\label{ass-GN}
\frac 1p-\frac kd \leq \frac 1r-\frac ld.
\end{equation}
Then the set
$$
K(a):=\{f\in C_{\textrm{w}}(\mathbb{R}_+;W^{k,p}_{\textrm{loc}}):\,\|f\|_{L^\infty([0,m],W^{k,p}(B_m))}+\|f\|_{C^\gamma([0,m],\mathbb{W}^{l,r}_m)}\leq a_m,\,m\in\mathbb{N}\}
$$
is a metrizable compact set in $C_{\textrm{w}}(\mathbb{R}_+;W^{k,p}_{\textrm{loc}})$.
\end{corollary}
\begin{proof} See Corollary B.2 in \cite{o5}.
\end{proof}

\begin{proposition}\label{cacoli} The Skorokhod representation theorem \ref{thm-skoro} holds for every tight sequence of probability measures defined on the $\sigma$-algebra generated by the following family of maps
$$
\{C_{\textrm{w}}(\mathbb{R}_+;W^{k,p}_{\textrm{loc}})\ni f \mapsto \langle\varphi,f(t)\rangle\in \mathbb{R}\}:\;\varphi\in\mathscr D(\mathbb{R}^d,\mathbb{R}^n),\,t\in [0,\infty).
$$
\end{proposition}
\begin{proof}
See Proposition B.3 in \cite{o5}.
\end{proof}

\section{Two analytic lemmata}\label{sec:analytic}

\begin{lemma}\label{lem-an-1}
Suppose that $(S,\mathcal{S},\mu)$ is a finite measure space with $\mu \geq 0$. Assume that a Borel measurable function $f:\mathbb{R}\to\mathbb{R}$  is of linear growth. Then the following functional
\begin{equation}\label{ineq_anal_1}
F(u):=\int_D f(u(x))\,\mu(dx),\; u\in L^2(D,\mu)
\end{equation}
is well defined.  Moreover, if $f$ is of $C^1$ class such that $f^\prime$ is of linear growth, then $F$ is also of $C^1$ class and
\begin{equation}\label{ineq_anal_2}
d_uF(v)=F^\prime(u)(v)=\int_D f^\prime(u(x))\,v(x)\, \mu(dx), \;u,v\in   L^2(D,\mu).
\end{equation}
\end{lemma}

\begin{proof}
The measurability and the linear growth of $f$ together with the assumption that $\mu(D)<\infty$ (what implies that $L^2(D,\mu)\subset L^1(D,\mu)$, imply that $F$ is well defined.

Let us fix for each $u\in L^2(D,\mu)$.
The linear growth of $f^\prime$ and the Cauchy-Schwartz inequality imply that  the functional $\Phi$ defined by the RHS of the inequality \ref{ineq_anal_2} is a bounded linear functional on $L^2(D,\mu)$. It is standard to  show that $F$ is Fr{\'e}chet differentiable at $u$ and that $d_F=\Phi$.
\end{proof}

\begin{lemma}\label{lem-an-2}
Suppose that an $L^2(D,\mu)$-valued  sequence $v^k$ converges weakly in $L^2(D,\mu)$ to $v\in L^2(D,\mu)$.  Suppose also that an $L^2(D,\mu)$-valued  sequence $u^k$ converges strongly in $L^2(D,\mu)$ to $u\in L^2(D,\mu)$. Then for any $\varphi \in L^\infty(D,\mu)$,
\begin{equation}\label{ineq-lem-an-2}
\int_D u^k(x)v^k(x)\varphi(x)\, dx \to  \int_D u(x)v(x)\varphi(x)\, dx .
\end{equation}

\end{lemma}

\begin{proof}  We have the following sequence of inequalities

\begin{eqnarray}\label{ineq-lem-an-2b}
&& \vert \int_D u^k(x)v^k(x)\varphi(x)\, dx -  \int_D u(x)v(x)\varphi(x)\, dx\vert\\
\nonumber
 &\leq & \vert \int_D \big(u^k(x)-u(x)\big)v^k(x)\varphi(x)\, dx\vert  + \vert   \int_D (v^k(x)-v(x))u(x) \varphi(x)\, dx\vert\\
\nonumber
 &\leq & \big( \int_D |u^k(x)-u(x)|^2 \, dx \big)^{1/2} \big( \int_D |v^k(x)|^2\,dx \big)^{1/2}\vert \varphi\vert_{L^\infty}
   \\
\nonumber
   &+& \vert   \int_D (v^k(x)-v(x))u(x) \varphi(x)\, dx\vert.
\end{eqnarray}
Since on the one hand by the first assumption the sequence $\big( \int_D |v^k(x)|^2\,dx \big)^{1/2}=\vert v^k\vert_{L^2}$ is bounded and $ \int_D (v^k(x)-v(x))u(x) \varphi(x)\, dx$ converges to $0$ and on the other  by the second  assumption $\int_D |u^k(x)-u(x)|^2 \, dx $ converges to $0$, the result follows by applying  \eqref{ineq-lem-an-2b}.
\end{proof}

The same proof as above applies to the following result.
\begin{lemma}\label{lem-an-3}
Suppose that an $L_{\textrm{loc}}^2(\mathbb{R}^d)$-valued  sequence $v^k$ converges weakly in $L_{\textrm{loc}}^2(\mathbb{R}^d)$ to $v\in L_{\textrm{loc}}^2(\mathbb{R}^d)$.  Suppose also that an $L_{\textrm{loc}}^2(\mathbb{R}^d)$-valued  sequence $u^k$ converges strongly in $L_{\textrm{loc}}^2(\mathbb{R}^d)$ to $u\in L_{\textrm{loc}}^2(\mathbb{R}^d)$. Then for any  function $\varphi \in L^\infty_{\textrm{comp}}(\mathbb{R}^d)$,
\begin{equation}\label{ineq-lem-an-3}
\int_D u^k(x)v^k(x)\varphi(x)\, dx \to  \int_D u(x)v(x)\varphi(x)\, dx .
\end{equation}
\end{lemma}
\begin{proof}
The set $D$ in the proof of Lemma \ref{lem-an-2} has be chosen in such a way that $\supp \varphi \subset D$.
\end{proof}

\begin{lemma}\label{lem-an-4}
Suppose that an $L^2(D,\mu)$-valued  bounded sequence $v^k$ converges weakly in $L^1(D,\mu)$ to $v\in L^2(D,\mu)$.
Then $v^k\to v$ weakly in $L^2(D,\mu)$.
\end{lemma}

\section{A measurability lemma}\label{sec:measurability}

Let $X$ be a separable Fr\'echet space (with a countable system of pseudonorms $(\|\cdot\|_k)_{k\in\mathbb{N}}$, let $X_k$ be separable Hilbert spaces and $i_k:X\to X_k$ linear mappings such that $\|i_k(x)\|_{X_k}=\|x\|_k$, $k\ge 1$. Let $\varphi_{k,j}\in X^*_k$, $j\in\mathbb{N}$ separate points of $X_k$. Then the mappings $(\varphi_{k,j}\circ i_k)_{k,j\in\mathbb{N}}$ generate the Borel $\sigma$-algebra on $X$.
\begin{proof}
See Appendix C in \cite{o5}.
\end{proof}

\begin{corollary}\label{aml1} There exists a countable system $\varphi_k\in\mathscr D$ such that  for every   $m\ge 0$ the mappings
$$
H^m_{\textrm{loc}}\ni h\mapsto\langle h,\varphi_k\rangle_{L^2}\in\mathbb{R},\qquad k\in\mathbb{N}
$$
generate the Borel $\sigma$-algebra on $H^m_{\textrm{loc}}$.
\end{corollary}

\end{document}